\documentclass[11pt, reqno]{amsart}
\usepackage{amsmath, amsthm, amscd, amsfonts, amssymb, graphicx, color, enumerate, xcolor,  mathrsfs, latexsym}
\allowdisplaybreaks
\usepackage[bookmarksnumbered, colorlinks, plainpages, backref]{hyperref}
\hypersetup{colorlinks=true,linkcolor=red, anchorcolor=green, citecolor=cyan, urlcolor=red, filecolor=magenta, pdftoolbar=true}

\definecolor{cupgreen}{rgb}{0,0.498,0.208}
\definecolor{cupblue}{rgb}{0,0,.5}
\definecolor{cupred}{rgb}{1,0.04,0}
\definecolor{cuppink}{rgb}{0.925,0,0.545}
\definecolor{cupmagenta}{rgb}{0.624,0.161,0.424}
\definecolor{cupbrown}{rgb}{0.71,0.212,0.133}

\definecolor{cupgreen}{rgb}{0,0,0}
\definecolor{cupblue}{rgb}{0,0,0}
\definecolor{cupred}{rgb}{0,0,0}
\definecolor{cuppink}{rgb}{0,0,0}
\definecolor{cupmagenta}{rgb}{0,0,0}
\definecolor{cupbrown}{rgb}{0,0,0}

\definecolor{TITLE}{rgb}{0.0,0.0,1.0}
\definecolor{AUTHOR1}{rgb}{0.00,0.59,0.00}
\definecolor{AUTHOR2}{rgb}{0.50,0.00,1.00}
\definecolor{SECTION}{rgb}{0.50,0.00,1.00}
\definecolor{FOOTTITLE}{rgb}{0.00,0.50,0.75}
\definecolor{THM}{rgb}{0.7,0.3,0.5}
\definecolor{SEC}{rgb}{0,0,1}

\makeatletter
\@namedef{subjclassname@2010}{%
\textup{2010} Mathematics Subject Classification}
\makeatother
\normalsize
\newtheorem{theorem}{{\color{THM} Theorem}}[section]
\newtheorem{proposition}[theorem]{{\color{THM}Proposition}}
\newtheorem{corollary}[theorem]{{\color{THM}Corollary}}
\theoremstyle{definition}

\newtheorem{example}[theorem]{{\color{THM}Example}}

\newtheorem{remark}[theorem]{{\color{THM}Remark}}
\numberwithin{equation}{section}

\newcommand{\D}{\mathcal D}
\newcommand{\M}{\mathcal M}

\newcommand{\K}{\mathcal L}

\newcommand{\NN}{\mathcal N}

\newcommand{\bea}{\begin{eqnarray*}}
\newcommand{\eea}{\end{eqnarray*}}
\usepackage{amscd}

\begin{document}
%-------------------------- Pleased do not change the following line-------------------------------------------
\noindent \textcolor[rgb]{0.99,0.00,0.00}{}\\[.5in]
\title [Lie higher derivations on generalized matrix algebras]{Lie higher derivations on generalized matrix algebras}
\author[ Moafian]{ F. Moafian}
%\address{ Department of Pure Mathematics and  Center of Excellence
%in Analysis on Algebraic Structures (CEAAS), Ferdowsi University of Mashhad, P.O. Box 1159, Mashhad 91775, Iran.}
%\email{\textcolor[rgb]{0.00,0.00,0.84}{vishki@um.ac.ir}}
\address{Department of Pure Mathematics, Ferdowsi University
of Mashhad, P.O. Box 1159, Mashhad 91775, Iran.}
\email{\textcolor[rgb]{0.00,0.00,0.84}{fahimeh.moafian@yahoo.com}}
\begin{abstract}
In this paper, at first the construction of Lie higher derivations and higher derivations on a generalized matrix algebra were characterized; then the conditions under which a Lie higher derivation on generalized matrix algebras is proper are provided.
Finally, the applications of the findings are discused. 
\end{abstract}
\maketitle
\section{\color{SEC}introduction}
Let us recall some basic facts related to (Lie) higher derivations on a general algebra. Let $A$ be a unital algebra, over a unital commutative ring {\bf R}, $\mathbb N$ be the set of all natural numbers and $\mathbb N_0=\mathbb N\cup \{0\}$. 
\begin{enumerate}[\hspace{1em}\rm (a)]
\item A sequence $\D=\{\D_k\}_{k\in\mathbb{N}_0}$ (with $\D_0=id_A$) of linear maps on $A$ is called a higher derivation if \[\D_k(xy)=\sum_{i+j=k}\D_i(x)\D_j(y),\]
 for all $x,y\in A$ and $k\in\mathbb N_0$.
\item A sequence $\K=\{\K_k\}_{k\in\mathbb{N}_0}$ (with $\K_0=id_A$) of linear maps on $A$ is called a Lie higher derivation if
\[\K_k([x,y])=\sum_{i+j=k}[\K_i(x),\K_j(y)],\]
for all $x,y\in A$ and $k\in\mathbb N_0$, where $[\cdot,\cdot]$ stands for a commutator defined by $[x,y]=xy-yx$. 
\end{enumerate}
Note that $\D_1$ (resp. $\K_1$) is a derivation (resp. Lie derivation) when $\{\D_k\}_{k\in\mathbb{N}_0}$ (resp. $\{\K_k\}_{k\in\mathbb{N}_0}$) is a higher derivation (resp. Lie higher derivation). 
Let $D$ (resp. $L$) be a derivation (resp. Lie derivation) on $A$, then $\D=\{\frac{D^k}{k!}\}_{k\in\mathbb{N}_0}$ (resp. $\K=\{\frac{L^k}{k!}\}_{k\in\mathbb{N}_0}$) is a higher derivation (resp. Lie higher derivation) on $A$, where $D^0=id_A$ (resp. $\K^0=id_A$), the identity mapping of A. These kind of higher derivations (resp. Lie higher derivations) are called ordinary higher derivations (resp. Lie higher derivations).
Trivially, every higher derivation is a Lie higher derivation, but the converse is not true, in general. If $\D=\{\D_k\}_{k\in\mathbb{N}_0}$ is a higher derivation on $A$ and
$\tau=\{\tau_k\}_{k\in\mathbb{N}}$ is a sequence of linear maps on $A$ which is center valued (i.e. $\tau_k(A)\subseteq Z(A)$= the center of $A$), then $\D+\tau$ is a Lie higher derivation if and only if $\tau$ vanishes at commutators, i.e. $\tau_k([x,y])=0,$ for all $x,y\in A$ and $k\in\mathbb N$. Lie higher derivations of this form are called proper Lie higher derivations.
%A problem that we are dealing with is studying those conditions on an algebra such that every Lie higher derivation on it is
%proper.
We say that an algebra $A$ has Lie higher derivation (LHD for short) property if every Lie higher derivation on it is proper. A main problem in the realm of Lie higher derivations is that, under what conditions a Lie higher derivation on an algebra is proper. Many authors have studied the problem for various algebras; see \cite{MME, FH, H, LS, N, Q, QH, WX, XW, XW1} and references therein.

%In  \cite{M},  Mirzavaziri discussed the structure of higher derivation on an general algebra. He showed that every higher derivation on an algebra can be derived as certain combinations of its derivations.
Han \cite{H} studied Lie-type higher derivations on operator algebras. He showed that every Lie (triple) higher derivation on some
classical operator algebras is proper.
Wei and Xiao \cite{WX} have examined innerness of higher derivations on triangular algebras. They also discussed Jordan higher derivations and nonlinear Lie higher derivations on a triangular algebra in \cite{XW} and \cite{XW1}, respectively. Qi and Hou \cite{QH} showed that every Lie higher derivation on a nest algebra is proper. Li and Shen \cite{LS} and also Qi \cite{Q} have extended the main result of \cite{QH} for a triangular algebra by providing some sufficient conditions under which a Lie higher derivation on a triangular algebra is proper.

In this paper we investigate the LHD property for a generalized matrix algebra. Generalized matrix algebras were first introduced by Sands \cite{S}. Here, we offer definition of a generalized matrix algebra. A Morita context  $(A, B, M, N, \Phi_{MN}, \Psi_{NM})$ consists of two unital algebras $A$, $B$, an $(A,B)-$module $M$, a $(B,A)-$module $N$, and two module homomorphisms  $\Phi_{MN}:M\otimes_B N\longrightarrow A$ and $\Psi_{NM}:N\otimes_A M\longrightarrow B$ satisfying the following commutative diagrams:
\begin{equation*}\label{Dia}\begin{CD}
M\otimes_B N\otimes_A M @ >\Phi_{MN}\otimes I_M >> A\otimes_A M\\
@ VV I_M\otimes\Psi_{NM} V @ VV\cong V\\
M\otimes_BB @ >\cong >> M
\end{CD}
\end{equation*}
and
\begin{equation*}\begin{CD}
N\otimes_AM\otimes_BN @ > \Psi_{NM}\otimes I_N >> B\otimes_BN\\
@ VV I_N\otimes\Phi_{MN} V @ VV\cong V\\
N\otimes_AA @ > \cong >>N.
\end{CD}
\end{equation*}
For a Morita context  $(A, B, M, N, \Phi_{MN},\Psi_{NM})$, the set
\[\mathcal{G}=\left(\begin{array}{cc}
A & M \\
N & B \\
\end{array}\right)=\Bigg\{\left(\begin{array}{cc}
a & m \\
n & b \\
\end{array}\right)\Big|\ a\in A,\ m\in M,\ n\in N,\ b\in B\Bigg\}\]
forms an algebra under the usual matrix operations, where at least one of two modules $M$ and $N$ is nonzero. The algebra $\mathcal G$ is called a generalized matrix algebra. In above definition if $N=0$, then $\mathcal{G}$ becomes the triangular algebra ${\rm Tri}(A,M,B)$, whose (Lie) derivations and its properties are extensively examined by Cheung \cite{Ch2}. 

Let $\mathcal{G}=\left(\begin{array}{cc}
A & M \\
N & B \\
\end{array}\right)$ be a generalized matrix algebra. We are dealing with various types of faithfullness. 
\begin{enumerate}
\item The $(A,B)-$module $M$ is called left (resp. right) faithful if $aM=\{0\}$  (resp.  $M b=\{0\}$) necessities $a=0$ (resp.  $b=0$), for all $a\in A$ (resp. $b\in B$). If $M$ is both left and right faithful it is called faithful. The left and right faithfulness of $N$ can be defined in a similar way.
\item The $(A,B)-$module $M$ is called strongly faithful if 
\begin{enumerate}[\hspace{1em}]
\item either $M$ is faithful as a right $B-$module and $am=0$ implies $a=0$ or $m=0$ for all $a\in A, m\in M$; or
\item $M$ is faithful as a left $A-$module and $mb=0$ implies $m=0$ or $b=0$ for all $m\in M, b\in B.$
\end{enumerate}
The strong faithfulness for $N$ can be defined similarly.
\item The generalized matrix algebra $\mathcal G$ is called weakly faithful if 
\begin{eqnarray}\label{weakly faithful}
\nonumber aM=\{0\}=Na \quad {\rm implies} \quad a=0,\\
Mb=\{0\}=bN \quad {\rm implies} \quad b=0.
\end{eqnarray}
\end{enumerate}
It is evident that if $M$ is strongly faithful then $M$ is faithful and either $A$ or $B$ has no zero devisors. It is also trivial that if either $M$ or $N$ is faithful, then $\mathcal G$ is weakly faithful.\\
It is worth mentioning that in the case $\mathcal G$ is a triangular algebra the weak faithfulness of $\mathcal G$ is nothing more than faithfulness of $M$.

By an standard argument one can check that the center $Z(\mathcal{G})$ of $\mathcal{G}$ is
\[Z(\mathcal{G})=\{a\oplus b|\ a\in Z(A), b\in Z(B),\ am=mb,\ na=bn \quad{\rm for\ all} \quad m\in M, n\in N\},\]
where $a\oplus b=\left(\begin{array}{cc}
a & 0 \\
0 & b \\
\end{array}\right)\in \mathcal{G}.$ 
Consider two natural projections $\pi_A:\mathcal{G}\longrightarrow A$ and $\pi_B:\mathcal{G}\longrightarrow B$ by
\[\pi_A:\left(\begin{array}{cc}
           a & m \\
           n & b \\
\end{array}\right)
\mapsto a\quad {\rm and}\quad  \pi_B:\left(\begin{array}{cc}
           a & m \\
           n & b \\
\end{array}\right)\mapsto b.\]
Clearly $\pi_A(Z(\mathcal{G}))\subseteq Z(A)$ and $\pi_B(Z(\mathcal{G}))\subseteq Z(B)$. Moreover, if $\mathcal G$ is weakly faithful then  $\pi_A(Z(\mathcal{G}))$  is isomrphic to  $\pi_B(Z(\mathcal{G}))$. More precisely, there exists a unique algebra isomorphism
\[\varphi:\pi_A(Z(\mathcal{G}))\longrightarrow \pi_B(Z(\mathcal{G}))\] such that $am=m\varphi(a)$ and $\varphi(a)n=na$ for all $m\in M$, $n\in N$; or equivalently, $a\oplus\varphi(a)\in Z(\mathcal G)$ for all $a\in A,$ (see \cite[Proposition 2.1]{Be} and 
\cite[Proposition 3]{Ch2}).\\

This paper is organized as follows; in section 2, we characterize the structure of Lie higher derivations and higher derivations on the generalized matrix algebra $\mathcal G$. The LHD property for the generalized matrix algebra $\mathcal G$ is investigated in section 3. In section 4, we offer some alternative sufficient conditions ensuring the LHD property for $\mathcal G$ (Theorems \ref{ideal}, \ref{domain}, \ref{strong}). We then come to our main result, Theorem \ref{main}, collecting some sufficient conditions ensuring the LHD property for a generalized matrix algebra.
Section 5 includes some applications of our conclusions to some main examples of a generalized matrix algebra such as: trivial generalized matrix algebras, triangular algebras, unital algebras with a nontrivial idempotent, the algebra $B(X)$ of operators on a Banach space $X$ and the full matrix algebra $M_n(A)$ on a unital algebra $A$. Since the proof of Theorems \ref{LHD} and \ref{HD} are too long we devote section 6 to them.

\section{\color{SEC}The structure of (Lie) higher derivations on $\mathcal G$}
We start this section with the following result of \cite{LW} which describes the structure of derivations and Lie derivations on a generalized matrix algebra.
\begin{proposition}[{\cite[Propositions  4.1, 4.2]{LW}}]\label{k=1}
Let $\mathcal{G}$ be a generalized matrix algebra.

$\bullet$ If  $A$ and $B$ are $2-$torsion free then a linear map $\K_1:\mathcal{G}\to\mathcal{G}$ is a Lie derivation if and only if it has the presentation
\[\K_1\left(\begin{array}{cc}
a & m \\
n & b \\
\end{array}\right)=\left(\begin{array}{cc}
\mathfrak p_{11}(a)+\mathfrak p_{12}(b)-mn_1-m_1n& am_1-m_1b+\mathfrak f_{13}(m) \\
n_1a-bn_1+\mathfrak g_{14}(n) & \mathfrak q_{11}(a)+\mathfrak q_{12}(b)+n_1m+nm_1 \\
\end{array}\right),\]
where $m_1\in M, n_1\in N$ and
$\mathfrak p_{11}:A\longrightarrow A,\ \mathfrak p_{12}:B\longrightarrow A,\ \mathfrak q_{11}:A\longrightarrow B,\ \mathfrak q_{12}:B\longrightarrow B,\ \mathfrak f_{13}:M\longrightarrow M,\ \mathfrak g_{14}:N\longrightarrow N$ are linear maps satisfying the following properties:
\begin{enumerate}[\hspace{1em}\rm (1)]
\item $\mathfrak p_{11}$ and  $\mathfrak q_{12}$ are Lie derivations.
\item $\mathfrak p_{12}([b,b'])=0, \mathfrak q_{11}([a,a'])=0.$
\item $\mathfrak p_{12}(B)\subseteq Z(A), \mathfrak q_{11}(A)\subseteq Z(B).$
\item $\mathfrak f_{13}(am)=\mathfrak p_{11}(a)m-m\mathfrak q_{11}(a)+a\mathfrak f_{13}(m)$,  $\mathfrak f_{13}(mb)=m\mathfrak q_{12}(b)-\mathfrak p_{12}(b)m+\mathfrak f_{13}(m)b$.
\item $\mathfrak g_{14}(na)=n\mathfrak p_{11}(a)-\mathfrak q_{11}(a)n+\mathfrak g_{14}(n)a$, $\mathfrak g_{14}(bn)=\mathfrak q_{12}(b)n-n\mathfrak p_{12}(b)+b\mathfrak g_{14}(n)$.
\item $\mathfrak p_{11}(mn)-\mathfrak p_{12}(nm)=m\mathfrak g_{14}(n)+\mathfrak f_{13}(m)n$,  $\mathfrak q_{12}(nm)-\mathfrak q_{11}(mn)=\mathfrak g_{14}(n)m+n\mathfrak f_{13}(m)$.
%\item $m_1:=f_{11}(1)=-f_{12}(1)$ and $n_1:=g_{11}(1)=-g_{12}(1)$ ,
%\item $g_{11}(a)=n_1a$, $f_{11}(a)=am_1$, $g_{12}(b)=-bn_1$ and $f_{12}(b)=-m_1b$,
%\item $g_{13}(m)=0$ and $f_{14}(n)=0$,
%\item $p_{13}(m)=-mn_1$, $q_{13}(m)=n_1m$, $p_{14}(n)=-m_1n$ and $q_{14}(n)=nm_1$,
\end{enumerate}
\vspace{0.5cm}
$\bullet$ A linear map $\D_1:\mathcal{G}\to\mathcal{G}$ is a derivation if and only if it has the presentation
\[\D_1\left(\begin{array}{cc}
a & m \\
n & b \\
\end{array}\right)=\left(\begin{array}{cc}
\mathtt p_{11}(a)-mn_1-m_1n& am_1-m_1b+\mathtt f_{13}(m) \\
n_1a-bn_1+\mathtt g_{14}(n) & \mathtt q_{12}(b)+n_1m+nm_1 \\
\end{array}\right),\]
where $m_1\in M, n_1\in N$ and
$\mathtt p_{11}:A\longrightarrow A,\ \mathtt p_{12}:B\longrightarrow A,\ \mathtt q_{11}:A\longrightarrow B,\ \mathtt q_{12}:B\longrightarrow B,\ \mathtt f_{13}:M\longrightarrow M,\ \mathtt g_{14}:N\longrightarrow N$ are linear maps satisfying the  following properties:
\begin{enumerate}[\hspace{1em}\rm (a)]
\item $\mathtt p_{11}$ and  $\mathtt q_{12}$ are  derivations.
%\item $m_1:=f_{11}(1)=-f_{12}(1)$ and $n_1:=g_{11}(1)=-g_{12}(1)$ ,
%\item $g_{11}(a)=n_1a$, $f_{11}(a)=am_1$, $g_{12}(b)=-bn_1$ and $f_{12}(b)=-m_1b$,
\item $\mathtt f_{13}(am)=\mathtt p_{11}(a)m+a\mathtt f_{13}(m)$,  $\mathtt f_{13}(mb)=m\mathtt q_{12}(b)+\mathtt f_{13}(m)b$.
\item $\mathtt g_{14}(na)=n\mathtt p_{11}(a)+\mathtt g_{14}(n)a$,  $\mathtt g_{14}(bn)=\mathtt q_{12}(b)n+b\mathtt g_{14}(n)$.
\item $\mathtt p_{11}(mn)=m\mathtt g_{14}(n)+\mathtt f_{13}(m)n$, $\mathtt q_{12}(nm)=\mathtt g_{14}(n)m+ n\mathtt f_{13}(m).$
%\item $g_{13}(m)=0$, $f_{14}(n)=0$, $p_{12}(b)=0$ and $q_{11}(a)=0$,
%\item $p_{13}(m)=-mn_1$, $q_{13}(m)=n_1m$, $p_{14}(n)=-m_1n$ and $q_{14}(n)=nm_1$.
\end{enumerate}
\end{proposition}
We are preparing to describe the structure of Lie higher and higher derivations on a generalized matrix algebra. We start with  fixing some notations which will be needed in the sequel.
\subsection{\color{SEC}Some notations}
Before we proceed for the result, for more convenience, we fix some notations. Throughout $\mathbb{N}$ stands for the natural numbers and $\mathbb{N}_0:=\mathbb{N}\cup\{0\}$. For each $k\in\mathbb N,$ we define $\eta_k$ and $\nu_k$ by
\[\eta_k:=\left\lbrace\begin{array}{ccc}
(k-1)/2 & ; & k \mbox{ is odd}\\
k/2 & ; & k \mbox{ is even}
\end{array}\right.\ {\rm and}\quad  \nu_k:=\left\lbrace\begin{array}{ccc}
(k-1)/2& ; & k  \mbox{ is odd}\\
k/2-1 & ; & k  \mbox{ is even}
\end{array}\right..\]

Let $k\in\mathbb N$, $\alpha_i, \beta_i \in\{1, 2, \cdots, k\}, (1\leq i\leq r)$ and $n_{\alpha_i}\in N, m_{\beta_i}\in M$ we define
\begin{enumerate}[\hspace{1em}]
 \item $(\alpha+\beta)_r:=\sum_{i=1}^r(\alpha_i+\beta_i),$
 \item $(n_{\alpha}m_{\beta})^r:=n_{\alpha_2}m_{\beta_2}\ldots n_{\alpha_r}m_{\beta_r},$ and $(n_{\alpha}m_{\beta})_r:=n_{\alpha_r}m_{\beta_r}\ldots n_{\alpha_2}m_{\beta_2},$
 \item $(m_{\beta}n_{\alpha})^r:=m_{\beta_2}n_{\alpha_2}\ldots m_{\beta_r}n_{\alpha_r},$ and $(m_{\beta}n_{\alpha})_r:=m_{\beta_r}n_{\alpha_r}\ldots m_{\beta_2}n_{\alpha_2}.$
 \end{enumerate}
We also define $\NN_k$ and $\M_k$ so that $\NN_1:=n_1,\ \M_1:=m_1,\ \NN_2:=n_2,\ \M_2:=m_2$ and for each $k\geq 3,$ by
\[\NN_k=\sum_{r=1}^{\nu_k}\sum_{(\alpha+\beta)_r+\gamma=k}(\prod_{\rho=1}^rn_{\alpha_{\rho}}m_{\beta_{\rho}}n_{\gamma})+n_k,\]
\[\M_k=\sum_{r=1}^{\nu_k}\sum_{(\alpha+\beta)_r+\gamma=k}(\prod_{{\rho}=1}^r m_{\beta_{\rho}}n_{\alpha_{\rho}}m_{\gamma})+m_k.\]
For example, for $k=3$ and $k=4$ we have:
\begin{enumerate}[\hspace{1em}]
\item $\NN_3=n_1m_1n_1+n_3,\, \M_3=m_1n_1m_1+m_3.$
\item $\NN_4=n_1m_1n_2+n_1m_2n_1+n_2m_1n_1+n_4,\,\M_4=m_1n_1m_2+m_1n_2m_1+m_2n_1m_1+m_4.$
\end{enumerate}

\subsection{\color{SEC}The structure of Lie higher derivations on $\mathcal{G}$}
It is easy to check that every sequence $\{\K_k\}_{k\in\mathbb N_0}$ of linear mappings on a generalized matrix algebra $\mathcal{G}=\left(\begin{array}{cc}
A & M \\
N & B \\
\end{array}\right)$ enjoys the presentation
{\small \[\K_k\left(\begin{array}{cc}\tag{$\textcolor[rgb]{0.0,0.0,0.0}\bigstar$}
a & m\\
n & b\\
\end{array}\right)=\left(\begin{array}{cc}
{\mathfrak p}_{1k}(a)+{\mathfrak p}_{2k}(b)+{\mathfrak p}_{3k}(m)+{\mathfrak p}_{4k}(n)& {\mathfrak f}_{1k}(a)+{\mathfrak f}_{2k}(b)+{\mathfrak f}_{3k}(m)+{\mathfrak f}_{4k}(n)\\
{\mathfrak g}_{1k}(a)+{\mathfrak g}_{2k}(b)+{\mathfrak g}_{3k}(m)+{\mathfrak g}_{4k}(n) & {\mathfrak q}_{1k}(a)+{\mathfrak q}_{2k}(b)+{\mathfrak q}_{3k}(m)+{\mathfrak q}_{4k}(n)\\
\end{array}\right)\]}
for each $k\in \mathbb{N}_0$, where the entries mappings
$\mathfrak p_{1k}:A\longrightarrow A,\ \mathfrak p_{2k}:B\longrightarrow A,\ \mathfrak p_{3k}:M\longrightarrow A,\ \mathfrak p_{4k}:N\longrightarrow A,\ \mathfrak q_{1k}:A\longrightarrow B,\ \mathfrak q_{2k}:B\longrightarrow B,\ \mathfrak q_{3k}:M\longrightarrow B,\ \mathfrak q_{4k}:N\longrightarrow B,\ \mathfrak f_{1k}:A\longrightarrow M,\ \mathfrak f_{2k}:B\longrightarrow M,\ \mathfrak f_{3k}:M\longrightarrow M,\ \mathfrak f_{4k}:N\longrightarrow M,\ \mathfrak g_{1k}:A\longrightarrow N,\ \mathfrak g_{2k}:B\longrightarrow N,\ \mathfrak g_{3k}:M\longrightarrow N$ and $ \mathfrak g_{4k}:N\longrightarrow N$ are linear. For each $k\in\mathbb N$ we set $m_k:=\mathfrak f_{1k}(1)$ and 
$n_k:=\mathfrak g_{1k}(1).$

Before proceeding for the structure of Lie higher derivations on $\mathcal{G}$ in Theorem \ref{LHD}, we need to fix some more notations and also define some auxiliary mappings by setting $P_0=P'_0:=id_{A}, P''_0:=0$, $Q_0=Q'_0:=id_{B}, Q''_0:=0$ and $p_1=p'_1=q''_1:=0, q_1=q'_1=p''_1:=0$.
Further for each $k\in\mathbb N,$ we set $P_k:=\mathfrak p_{1k}+p_k,\ P'_k:=\mathfrak p_{1k}+p'_k,\ P''_k:=\mathfrak p_{2k}-p''_k,\ Q_k:=\mathfrak q_{2k}+q_k,\ Q'_k:=\mathfrak q_{2k}+q'_k$ and $Q''_k:=\mathfrak q_{1k}-q''_k$ where for each $k\geq 2$
\begin{eqnarray*}
p_k(a)&:=&\sum_{r=1}^{\eta_k}\sum_{i+(\alpha+\beta)_r=k,\, i\leq k-2}(P_i(a)m_{\beta_1}-m_{\beta_1}Q''_i(a))n_{\alpha_1}(m_{\beta}n_{\alpha})^r,\\
p'_k(a)&:=&\sum_{r=1}^{\eta_k}\sum_{i+(\alpha+\beta)_r=k,\, i\leq k-2}(m_{\beta}n_{\alpha})_r m_{\beta_1}(n_{\alpha_1}P'_i(a)-Q''_i(a)n_{\alpha_1}),\\
p''_k(b)&:=&\sum_{r=1}^{\eta_k}\sum_{i+(\alpha+\beta)_r=k,\, i\leq k-2}(m_{\beta_1}Q_i(b)-P''_i(b)m_{\beta_1})n_{\alpha_1}(m_{\beta}n_{\alpha})^r,\\
&=&\sum_{r=1}^{\eta_k}\sum_{i+(\alpha+\beta)_r=k,\, i\leq k-2}(m_{\beta}n_{\alpha})_r m_{\beta_1}(Q'_i(b)n_{\alpha_1}-n_{\alpha_1}P''_i(b)),\\
q_k(b)&:=&\sum_{r=1}^{\eta_k}\sum_{i+(\alpha+\beta)_r=k,\, i\leq k-2}(n_{\alpha}m_{\beta})_r n_{\alpha_1}(m_{\beta_1}Q_i(b)-P''_i(b)m_{\beta_1}),\\
q'_k(b)&:=&\sum_{r=1}^{\eta_k}\sum_{i+(\alpha+\beta)_r=k,\, i\leq k-2} (Q'_i(b)n_{\alpha_1}-n_{\alpha_1}P''_i(b))m_{\beta_1}(n_{\alpha}m_{\beta})^r,\\
q''_k(a)&:=&\sum_{r=1}^{\eta_k}\sum_{i+(\alpha+\beta)_r=k,\, i\leq k-2}(n_{\alpha}m_{\beta})_r n_{\alpha_1}(P_i(a)m_{\beta_1}-m_{\beta_1}Q''_i(a))\\
&=&\sum_{r=1}^{\eta_k}\sum_{i+(\alpha+\beta)_r=k,\, i\leq k-2}(n_{\alpha_1}P'_i(a)-Q''_i(a)n_{\alpha_1})m_{\beta_1}(n_{\alpha}m_{\beta})^r.
\end{eqnarray*}
For example; for  $k=2$ we get;
\[p_2(a)=am_1n_1,\ p'_2(a)=m_1n_1a,\ p''_2(b)=m_1bn_1,\ q_2(b)=n_1m_1b,\ q'_2(b)=bn_1m_1\ and \ q''_2(a)=n_1am_1.\]
Similarly, for $k=3$ one can check that
\begin{eqnarray*}
p_3(a)&=&am_1n_2+am_2n_1+P_1(a)m_1n_1-m_1Q''_1(a)n_1\\
p'_3(a)&=&m_1n_2a+m_2n_1a+m_1n_1P_1(a)-m_1Q''_1(a)n_1,\\
p''_3(b)&=&m_1bn_2+m_2bn_1+m_1Q_1(b)n_1-P''_1(b)m_1n_1\\
&=&m_1bn_2+m_2bn_1+m_1Q'_1(b)n_1-m_1n_1P''_1(b),\\
q_3(b)&=&n_1m_2b+n_2m_1b+n_1m_1Q_1(b)-n_1P''_1(b)m_1\\
q'_3(b)&=&bn_1m_2+bn_2m_1+Q_1(b)n_1m_1-n_1P''_1(b)m_1\\
q''_3(a)&=&n_1am_2+n_2am_1+n_1P_1(a)m_1-n_1m_1Q''_1(a)\\
&=&n_1am_2+n_2am_1+n_1P'_1(a)m_1-Q''_1(a)n_1m_1.
\end{eqnarray*}

Now, we are ready to present the following result describing the structure of Lie higher derivations on the generalized matrix algebra $\mathcal{G}$.
\begin{theorem}\label{LHD}
Let $\mathcal{G}=\left(\begin{array}{cc}
A & M \\
N & B \\
\end{array}\right)$ be a generalized matrix algebra such that $A, B$ are $2$-torsion free. Then a sequence $\K=\{\K_k\}_{k\in\mathbb{N}_0}:\mathcal{G}\longrightarrow\mathcal{G}$ of linear mappings (as presented in ($\bigstar$)) is a Lie higher derivation if and only if
\begin{enumerate}[\hspace{1em}\rm (1)]
\item $\{P_k\}_{k\in\mathbb{N}_0}, \{P'_k\}_{k\in\mathbb{N}_0}$ are Lie higher derivations on $A$,  $\{Q_k\}_{k\in\mathbb{N}_0}, \{Q'_k\}_{k\in\mathbb{N}_0}$ are Lie higher derivations on $B,$  $Q''_k(A)\subseteq Z(B), P''_k(B)\subseteq Z(A),$ and $Q''_k([a,a'])=0, P''_k([b,b'])=0$, for all $k\in\mathbb N.$ \\
\item $\mathfrak {g_1}_k(a)=\sum_{i+j=k,\, j\neq k}(n_iP'_j(a)-Q''_j(a)n_i)$ and $\mathfrak f_{1k}(a)=\sum_{i+j=k,\, j\neq k}(P_j(a)m_i-m_iQ''_j(a))$,\\
\item $\mathfrak {g_2}_k(b)=-\sum_{i+j=k,\, j\neq k}(Q'_j(b)n_i-n_iP''_j(b))$ and $\mathfrak f_{2k}(b)=-\sum_{i+j=k,\, j\neq k}(m_iQ_j(b)-P''_j(b)m_i)$,\\
\item $\mathfrak f_{3k}(am)=\sum_{i+j=k}\big(P_i(a)\mathfrak f_{3j}(m)-\mathfrak f_{3j}(m)Q''_i(a)\big)$,\\
\item $\mathfrak f_{3k}(mb)=\sum_{i+j=k}\big(\mathfrak f_{3j}(m)Q_i(b)-P''_i(b)\mathfrak f_{3j}(m)\big)$,\\
\item $\mathfrak g_{4k}(na)=\sum_{i+j=k}\big(\mathfrak g_{4j}(n)P'_i(a)-Q''_i(a)\mathfrak g_{4j}(n)\big)$,\\
\item $\mathfrak g_{4k}(bn)=\sum_{i+j=k}\big(Q'_i(b)\mathfrak g_{4j}(n)-\mathfrak g_{4j}(n)P''_i(b)\big)$,\\
\item $\mathfrak g_{3k}(m)=-\sum_{i+j+r=k}\NN_i\mathfrak f_{3r}(m)\NN_j$ and $\mathfrak f_{4k}(n)=-\sum_{i+j+r=k}\M_i\mathfrak g_{4r}(n)\M_j$,\\
\item $\mathfrak p_{3k}(m)=\sum_{i+j=k}-\mathfrak f_{3i}(m)\NN_j$ and $\mathfrak q_{3k}(m)=\sum_{i+j=k}\NN_j\mathfrak f_{3i}(m)$,\\
\item $\mathfrak p_{4k}(n)=\sum_{i+j=k}-\M_j\mathfrak g_{4i}(n)$ and $\mathfrak q_{4k}(n)=\sum_{i+j=k}\mathfrak g_{4i}(n)\M_j$,\\
\item $\mathfrak p_{1k}(mn)-\mathfrak p_{2k}(nm)=\sum_{i+j=k}\big(\mathfrak p_{3i}(m)\mathfrak p_{4j}(n) +\mathfrak f_{3i}(m)\mathfrak g_{4j}(n)-\mathfrak p_{4j}(n)\mathfrak p_{3i}(m)-\mathfrak f_{4j}(n)\mathfrak g_{3i}(m)\big),$\\
\item $\mathfrak q_{1k}(mn)-\mathfrak q_{2k}(nm)=\sum_{i+j=k}\big(\mathfrak g_{3i}(m)\mathfrak f_{4j}(n)+\mathfrak q_{3i}(m)\mathfrak q_{4j}(n)-\mathfrak g_{4j}(n)\mathfrak f_{3i}(m)-\mathfrak q_{4j}(n)\mathfrak q_{3i}(m)\big).$
\end{enumerate}
\end{theorem}
\subsection{\color{SEC}The structure of higher derivations on $\mathcal{G}$}
Similar to $(\bigstar)$ let
$\D=\{\D_k\}_{k\in\mathbb{N}_0}:\mathcal{G}\longrightarrow\mathcal{G}$ be a sequence of linear mappings with the following presentation
\[{\small\D_k\left(\begin{array}{cc}\tag{$\textcolor[rgb]{0.0,0.0,0.0}\clubsuit$}
a & m\\
n & b \\
\end{array}\right)=\left(\begin{array}{cc}
\mathtt p_{1k}(a)+\mathtt p_{2k}(b)+\mathtt p_{3k}(m)+\mathtt p_{4k}(n) & \mathtt f_{1k}(a)+\mathtt f_{2k}(b)+\mathtt f_{3k}(m)+\mathtt f_{4k}(n)\\
\mathtt g_{1k}(a)+\mathtt g_{2k}(b)+\mathtt g_{3k}(m)+\mathtt g_{4k}(n) & \mathtt q_{1k}(a)+\mathtt q_{2k}(b)+\mathtt q_{3k}(m)+\mathtt q_{4k}(n)\\
\end{array}\right),}\]
whose entries maps are linear. As before for each $k\in\mathbb N_0$ we set $m_k:=\mathtt f_{1k}(1)$ and $n_k:=\mathtt g_{1k}(1).$

Before proceeding for the structure of higher derivations on $\mathcal{G}$ in Theorem \ref{HD}, we need to fix some more notations and also define some auxiliary mappings by setting ${\mathsf P}_0={\mathsf P}'_0:=id_{A}$, ${\mathsf Q}_0={\mathsf Q}'_0:=id_{B}$, and ${\mathsf p}_1={\mathsf p}'_1:=0$, ${\mathsf q}_1={\mathsf q}'_1:=0.$ Further for each $k\in\mathbb N,$ we set ${\mathsf P}_k:=\mathtt p_{1k}+{\mathsf p}_k$, ${\mathsf P}'_k:=\mathtt p_{1k}+{\mathsf p}'_k$,
${\mathsf Q}_k:=\mathtt q_{2k}+{\mathsf q}_k$, ${\mathsf Q}'_k:=\mathtt q_{2k}+{\mathsf q}'_k,$ where for for each $k\geq 2$
\begin{eqnarray*}
{\mathsf p}_k(a)&:=&\sum_{r=1}^{\eta_k}\sum_{i+(\alpha+\beta)_r=k,\, i\leq k-2}{\mathsf P}_i(a)m_{\beta_1}n_{\alpha_1}(m_{\beta}n_{\alpha})^r,\\
{\mathsf p}'_k(a)&:=&\sum_{r=1}^{\eta_k}\sum_{i+(\alpha+\beta)_r=k,\, i\leq k-2}m_{\beta_1}n_{\alpha_1}(m_{\beta}n_{\alpha})^r{\mathsf P}'_i(a),\\
{\mathsf q}_k(b)&:=&\sum_{r=1}^{\eta_k}\sum_{i+(\alpha+\beta)_r=k,\, i\leq k-2}(n_{\alpha}m_{\beta})_r n_{\alpha_1}m_{\beta_1}{\mathsf Q}_i(b),\\
{\mathsf q}'_k(b)&:=&\sum_{r=1}^{\eta_k}\sum_{i+(\alpha+\beta)_r=k,\, i\leq k-2}{\mathsf Q}'_i(b)(n_{\alpha}m_{\beta})_r n_{\alpha_1}m_{\beta_1}.\\
\end{eqnarray*}
In particular for the case $k=2$ we get
\[ {\mathsf p}_2(a)=am_1n_1,\ {\mathsf p}'_2(a)=m_1n_1a,\ {\mathsf q}_2(b)=n_1m_1b\ and\ {\mathsf q}'_2(b)=bn_1m_1.\]
Similarly for the case $k=3$ it is easy to check that
\begin{eqnarray*}
{\mathsf p}_3(a)&=&am_1n_2+am_2n_1+{\mathsf P}_1(a)m_1n_1,\\
{\mathsf p}'_3(a)&=&m_1n_2a+m_2n_1a+m_1n_1{\mathsf P}'_1(a),\\
{\mathsf q}_3(b)&=&n_1m_2b+n_2m_1b+n_1m_1{\mathsf Q}_1(b),\\
{\mathsf q}'_3(b)&=&bn_1m_2+bn_2m_1+{\mathsf Q}'_1(b)n_1m_1.
\end{eqnarray*}
Parallel to Theorem \ref{LHD}, in the following result we characterize the structure of higher derivations on the generalized matrix algebra $\mathcal G$.
\begin{theorem}\label{HD}
Let $\mathcal{G}=\left(\begin{array}{cc}
A & M \\
N & B \\
\end{array}\right)$ be a generalized matrix algebra. Then a sequence $\D=\{\D_k\}_{k\in\mathbb{N}_0}:\mathcal{G}\longrightarrow\mathcal{G}$ of linear mappings (as presented in $\clubsuit$) is a higher derivation if and only if
\begin{enumerate}[\hspace{1em}\rm (a)]
\item $\{{\mathsf P}_k\}_{k\in\mathbb{N}_0}, \{{\mathsf P}'_k\}_{k\in\mathbb{N}_0}$ are higher derivations on $A$ and $\{{\mathsf Q}_k\}_{k\in\mathbb{N}_0}, \{{\mathsf Q}'_k\}_{k\in\mathbb{N}_0}$ are higher derivations on $B.$\\
\item $\mathtt g_{1k}(a)=\sum_{i+j=k, j\neq k}n_i{\mathsf P}'_j(a)$ and $\mathtt f_{1k}(a)=\sum_{i+j=k, j\neq k}{\mathsf P}_j(a)m_i$,\\
\item $\mathtt g_{2k}(b)=-\sum_{i+j=k, j\neq k}{\mathsf Q}'_j(b)n_i$ and $\mathtt f_{2k}(b)=-\sum_{i+j=k, j\neq k}m_i{\mathsf Q}_j(b)$,\\
\item $\mathtt f_{3k}(am)=\sum_{i+j=k}{\mathsf P}_i(a)\mathtt f_{3j}(m)$ and $\mathtt f_{3k}(mb)=\sum_{i+j=k}\mathtt f_{3j}(m){\mathsf Q}_i(b)$,\\
\item $\mathtt g_{4k}(na)=\sum_{i+j=k}\mathtt g_{4j}(n){\mathsf P}'_i(a)$ and $\mathtt g_{4k}(bn)=\sum_{i+j=k}{\mathsf Q}'_i(b)\mathtt g_{4j}(n)$,\\
\item $\mathtt g_{3k}(m)=-\sum_{i+j+r=k}\NN_i\mathtt f_{3r}(m)\NN_j$ and $\mathtt f_{4k}(n)=-\sum_{i+j+r=k}\M_i\mathtt g_{4r}(n)\M_j$,\\
\item $\mathtt p_{3k}(m)=-\sum_{i+j=k}\mathtt f_{3i}(m)\NN_j$ and $\mathtt q_{3k}(m)=\sum_{i+j=k}\NN_j\mathtt f_{3i}(m)$,\\
\item $\mathtt p_{4k}(n)=-\sum_{i+j=k}\M_j\mathtt g_{4i}(n)$ and $\mathtt q_{4k}(n)=\sum_{i+j=k}\mathtt g_{4i}(n)\M_j$,\\
\item $\mathtt q_{1k}(a)=\sum_{r=1}^{\eta_k}\sum_{i+(\alpha+\beta)_r=k,\, i\leq k-2}(n_{\alpha}m_{\beta})_r n_{\alpha_1}{\mathsf P}_i(a)m_{\beta_1}\\
\hspace*{1.1cm}=\sum_{r=1}^{\eta_k}\sum_{i+(\alpha+\beta)_r=k,\, i\leq k-2}n_{\alpha_1}{\mathsf P}'_i(a)m_{\beta_1}(n_{\alpha}m_{\beta})^r,$\\
\item $\mathtt p_{2k}(b)=\sum_{r=1}^{\eta_k}\sum_{i+(\alpha+\beta)_r=k,\, i\leq k-2}m_{\beta_1}{\mathsf Q}_i(b)n_{\alpha_1}(m_{\beta}n_{\alpha})^r\\
\hspace*{1.07cm}=\sum_{r=1}^{\eta_k}\sum_{i+(\alpha+\beta)_r=k,\, i\leq k-2}(m_{\beta}n_{\alpha})_r m_{\beta_1}{\mathsf Q}'_i(b)n_{\alpha_1},$\\
\item $\mathtt p_{1k}(mn)=\sum_{i+j=k}\big(\mathtt p_{3i}(m)\mathtt p_{4j}(n)+\mathtt f_{3i}(m)\mathtt g_{4j}(n)\big),$ \\
\item $\mathtt q_{1k}(mn)=\sum_{i+j=k}\big(\mathtt g_{3i}(m)\mathtt  f_{4j}(n)+\mathtt q_{3i}(m) \mathtt q_{4j}(n)\big),$\\
\item $\mathtt q_{2k}(nm)=\sum_{i+j=k}\big(\mathtt g_{4i}(n)\mathtt  f_{3j}(m)+\mathtt q_{4i}(n) \mathtt q_{3j}(m)\big),$\\
\item $\mathtt p_{2k}(nm)=\sum_{i+j=k}\big(\mathtt p_{4i}(n)\mathtt p_{3j}(m)+\mathtt f_{4i}(n)\mathtt g_{3j}(m)\big).$
\end{enumerate}
\end{theorem}
\subsection{\color{SEC} The center valued mappings on $\mathcal G$} 
In the next result we characterize the structure of center valued maps on $\mathcal G$ vanishing at commutators of $\mathcal G.$
\begin{proposition}\label{center}
A sequence  $\tau=\{\tau_k\}_{k\in\mathbb{N}}$ of linear maps  on $\mathcal{G}$ is center valued and vanishes at commutators if and only if for each $k\in\mathbb N$ the map $\tau_k$ has the presentation 
\[\tau_k\left(\begin{array}{cc}
a & m\\
n & b
\end{array}\right)=\left(\begin{array}{cc}
\ell_k(a)+P''_k(b) &  \\
  & Q''_k(a)+\ell'_k(b)
\end{array}\right),\]
 where $\ell_k:A\longrightarrow Z(A)$, $P''_k:B\longrightarrow Z(A)$, $Q''_k:A\longrightarrow Z(B)$ and $\ell'_k:B\longrightarrow Z(B)$ are linear maps vanishing at commutators and satisfying the following properties:
\begin{enumerate}[\hspace{1em}\rm (i)]
\item $\ell_k(a)\oplus Q''_k(a)\in Z(\mathcal G)$ and  $P''_k(b)\oplus \ell'_k(b)\in Z(\mathcal G),$ for all $a\in A, b\in B$ and $k\in\mathbb{N}$.
\item $\ell_k(mn)=P''_k(nm)$ and $\ell'_k(nm)=Q''_k(mn),$  for all $ m \in M, n\in N$ and $k\in\mathbb{N}$.
\end{enumerate}
\end{proposition}
%%%%%%%%%%%%%%%%%%%%%%%%%%%%%%%%%%%%%%%
\section{\color{SEC} Proper Lie higher derivations}
Hereinafter, suppose that the modules $M$ and $N$ appeared in definition of the generalized matrix algebras are $2-$torion free; ($M$ is said to be $2-$torsion free if $2m=0$ implies $m=0$ for all $m\in M$).

According to the Cheung's method {\cite[Theorem 6]{Ch2}}, in the following theorem we give a necessary and sufficient condition under which Lie higher derivations on the generalized matrix algebra $\mathcal G$ are proper.
\begin{theorem}\label{Asasi}
Let $\mathcal{G}$ be a generalized matrix algebra. A Lie higher derivation $\K $ on $\mathcal{G}$ of the form {\rm ($\bigstar$)} is proper if and only if there exist two sequences of linear mappings $\{\ell_k\}_{k\in\mathbb{N}}:A \longrightarrow Z(A)$ and 
$\{\ell'_k\}_{k\in\mathbb{N}}:B \longrightarrow Z(B)$ satisfying the following three properties:
\begin{enumerate}[\hspace{1em}\rm ($A$)]
\item $\{P_k-\ell_k\}_{k\in\mathbb{N}}$ and $\{Q_k-\ell'_k\}_{k\in\mathbb{N}}$ are higher derivations on $A$ and $B$, respectively.
\item $\ell_k(a)\oplus Q''_k(a)\in Z(\mathcal G)$ and  $P''_k(b)\oplus \ell'_k(b)\in Z(\mathcal G),$ for all $a\in A, b\in B$ and $k\in\mathbb{N}$.
\item $\ell_k(mn)=P''_k(nm)$ and $\ell'_k(nm)=Q''_k(mn),$ for all $ m \in M, n\in N$ and $k\in\mathbb{N}$.
\end{enumerate}
\end{theorem}
\begin{proof}
For sufficiency by induction on $k$, we know that for $k=1$ the result is true by \cite{Ch2}. Now let the result holds for any integer less than $k$, we prove this for $k$. By using induction hypothesis and Theorems \ref{LHD} and \ref{HD}, like step $1$ appeared in \cite{DW}, without loss of generality we can consider structure of $\K_k$ as
\[\K_k\left(\begin{array}{cc}
a & m\\
n & b
\end{array}\right)=\left(\begin{array}{cc}
\mathfrak p_{1k}(a)+\mathfrak p_{2k}(b) & \mathfrak f_{1k}(a)+\mathfrak f_{2k}(b)+\mathfrak f_{3k}(m)\\
\mathfrak g_{1k}(a)+\mathfrak g_{2k}(b)+\mathfrak g_{4k}(n) & \mathfrak q_{1k}(a)+\mathfrak q_{2k}(b)
\end{array}\right)\]
where $\mathfrak p_{1k}, \mathfrak p_{2k}, \mathfrak q_{1k}, \mathfrak q_{2k}, \mathfrak f_{3k}$ and $\mathfrak g_{4k}$ have the properties $(1), (2), (6), (7), (8)$ and $(9)$ in Lemma \ref{LHD}. Replace $\mathfrak p_{1k}, \mathfrak p_{2k}, \mathfrak q_{1k}, \mathfrak q_{2k}$ with
$P_k-p_k, Q_k-q_k, P''_k+p''_k, Q''_k+q''_k$ respectively, then we have
\small{\begin{eqnarray}\label{3.1}
\K_k\left(\begin{array}{cc}
a & m\\
n & b
\end{array}\right)\hspace*{-.3cm}&=&\hspace*{-.3cm}\left(\begin{array}{cc}
P_k(a)-p_k(a)+P''_k(b)+p''_k(b) & \mathfrak f_{1k}(a)+\mathfrak f_{2k}(b)+\mathfrak f_{3k}(m)\\
\mathfrak g_{1k}(a)+\mathfrak g_{2k}(b)+\mathfrak g_{4k}(n) & Q''_k(a)+q''_k(a)+Q_k(b)-q_k(b)
\end{array}\right),
\end{eqnarray}}
By the induction hypothesis as $\K_i$ is proper for all $i<k$ we may write $p_k, p'_k, p''_k, q_k, q'_k, q''_k, \mathfrak f_{1k}, \mathfrak f_{2k}, \mathfrak g_{1k}$ and $\mathfrak g_{2k}$ as follows
\[p_k(a)=\sum_{r=1}^{\eta_k}\sum_{i+(\alpha+\beta)_r=k,\, i\leq k-2}\big(P_i(a)-\ell_i(a)\big)m_{\beta_1}n_{\alpha_1}\ldots m_{\beta_r}n_{\alpha_r},\]
\[p'_k(a)=\sum_{r=1}^{\eta_k}\sum_{i+(\alpha+\beta)_r=k,\, i\leq k-2}m_{\beta_r}n_{\alpha_r}\ldots m_{\beta_1}n_{\alpha_1}\big(P'_i(a)-\ell_i(a)\big),\]
\begin{eqnarray*}
p''_k(b)&=&\sum_{r=1}^{\eta_k}\sum_{i+(\alpha+\beta)_r=k,\, i\leq k-2}m_{\beta_1}\big(Q_i(b)-\ell'_i(b)\big)n_{\alpha_1}\ldots m_{\beta_r}n_{\alpha_r}\\
&=&\sum_{r=1}^{\eta_k}\sum_{i+(\alpha+\beta)_r=k,\, i\leq k-2}m_{\beta_r}n_{\alpha_r}\ldots m_{\beta_1}\big(Q'_i(b)-\ell'_i(b)\big)n_{\alpha_1},
\end{eqnarray*}
\[q_k(b)=\sum_{r=1}^{\eta_k}\sum_{i+(\alpha+\beta)_r=k,\, i\leq k-2}n_{\alpha_r}m_{\beta_r}\ldots n_{\alpha_1}m_{\beta_1}\big(Q_i(b)-\ell'_i(b)\big),\]
\[q'_k(b)=\sum_{r=1}^{\eta_k}\sum_{i+(\alpha+\beta)_r=k,\, i\leq k-2}\big(Q'_i(b)-\ell'_i(b)\big)n_{\alpha_1}m_{\beta_1}\ldots n_{\alpha_r}m_{\beta_r},\]
\begin{eqnarray*}
q''_k(a)&=&\sum_{r=1}^{\eta_k}\sum_{i+(\alpha+\beta)_r=k,\, i\leq k-2}n_{\alpha_r}m_{\beta_r}\ldots n_{\alpha_1}\big(P_i(a)-\ell_i(a)\big)m_{\beta_1}\\
&=&\sum_{r=1}^{\eta_k}\sum_{i+(\alpha+\beta)_r=k,\, i\leq k-2}n_{\alpha_1}\big(P'_i(a)-\ell_i(a)\big)m_{\beta_1}\ldots n_{\alpha_r}m_{\beta_r}
\end{eqnarray*}
\[\mathfrak f_{1k}(a)=\sum_{i+j=k, i\neq k}(P_i(a)m_j-m_jQ''_i(a))=\sum_{i+j=k, i\neq k}\big(P_i(a)-\ell_i(a)\big)m_j,\]
\[\mathfrak f_{2k}(b)=-\sum_{i+j=k, i\neq k}(m_jQ_i(b)-P''_i(b)m_j)=-\sum_{i+j=k, i\neq k}m_j\big(Q_i(b)-\ell'_i(b)\big),\]
\[\mathfrak g_{1k}(a)=\sum_{i+j=k, i\neq k}(n_jP'_i(a)-Q''_i(a)n_j)=\sum_{i+j=k, i\neq k}n_j\big(P'_i(a)-\ell_i(a)\big),\]
\[\mathfrak g_{2k}(b)=-\sum_{i+j=k, i\neq k}(Q'_i(b)n_j-n_jP''_i(b))=-\sum_{i+j=k, i\neq k}\big(Q'_i(b)-\ell'_i(b)\big)n_j,\]
i.e. $p_k(a), p'_k(a), p''_k(b), q_k(b), q'_k(b), q''_k(a), \mathfrak f_{1k}(a), \mathfrak f_{2k}(b), \mathfrak g_{1k}(a)$ and $\mathfrak g_{2k}(b)$ are the same as those appeared in the structure of higher derivations in Theorem \ref{HD}. So we can present \eqref{3.1} to the simpler form
\begin{eqnarray*}
\K_k\left(\begin{array}{cc}
a & m\\
n & b
\end{array}\right)&=&\left(\begin{array}{cc}
P_k(a)+P''_k(b)& \mathfrak f_{3k}(m)\\
\mathfrak g_{4k}(n) & Q''_k(a)+Q_k(b)
\end{array}\right).
\end{eqnarray*}
Set
\[\D_k\left(\begin{array}{cc}
a & m\\
n & b
\end{array}\right)=\left(\begin{array}{cc}
P_k(a)-\ell_k(a) & \mathfrak f_{3k}(m)\\
\mathfrak g_{4k}(n) & Q_k(b)-\ell'_k(b)
\end{array}\right),\]
and
\[\tau_k\left(\begin{array}{cc}
a & m\\
n & b
\end{array}\right)=\left(\begin{array}{cc}
\ell_k(a)+P''_k(b)& \\
 & Q''_k(a)+\ell'_k(b)
\end{array}\right).\]
For more convenience set
\begin{enumerate}[\hspace{1em}\rm (i)]
\item $\mathsf P_k(a)=P_k(a)-\ell_k(a),\ \mathsf Q_k(b)=Q_k(b)-\ell'_k(b)$
\item $\gamma_k(a, b)=\ell_k(a)+P''_k(b)$ and
\item $\gamma'_k(a, b)=Q''_k(a)+\ell'_k(b)$.
\end{enumerate}
Considering the above relations we have
\[\K_k\left(\begin{array}{cc}
a & m\\
n & b
\end{array}\right)=\left(\begin{array}{cc}
\mathsf P_k(a)+\gamma_k(a, b) & \mathfrak f_{3k}(m)\\
\mathfrak g_{4k}(n) & \mathsf Q_k(b)+\gamma'_k(a, b)
\end{array}\right).\]
Apply $\K_k$ on commutator
$\left[\left(\begin{array}{cc}
0 & 0\\
n & 1
\end{array}\right),\left(\begin{array}{cc}
0 & m\\
0 & -nm
\end{array}\right)\right]$, we get:
\begin{equation}\label{3.4}
P_k(mn)=\sum_{i+j=k}\mathfrak f_{3i}(m)\mathfrak g_{4j}(n)-\gamma_k(mn,-nm),
\end{equation}
\begin{equation}\label{3.5}
Q_k(nm)=\sum_{i+j=k}\mathfrak g_{4j}(n)\mathfrak f_{3i}(m)-\gamma'_k(mn,-nm).
\end{equation}
From the assumption $(C)$, since $\ell_k(mn)=P''_k(nm)$ and $\ell'_k(nm)=Q''_k(mn)$, then $\gamma_k(mn,-nm)=0$ and $\gamma'_k(mn,-nm)=0$, it follows that
\[\mathsf P_k(mn)=\sum_{i+j=k}\mathfrak f_{3i}(m)\mathfrak g_{4j}(n)\, \, \,  and \, \, \, 
\mathsf Q_k(nm)=\sum_{i+j=k}\mathfrak g_{4j}(n)\mathfrak f_{3i}(m)\]
for all $ m \in M, n\in N$ and $k\in\mathbb{N}$. A direct verification reveals that $\D=\{\D_k\}_{k\in\mathbb{N}}$ is a higher derivation and $\tau=\{\tau_k\}_{k\in\mathbb{N}}$ is a sequence of center valued  maps.

For necessity, let $\K$ be proper i.e. $\K=\D+\tau$ for some higher derivation $\D$ and a sequence of center valued maps $\tau$. Applying the presentations ($\bigstar$), ($\clubsuit$) for $\K$ and $\D$, respectively, we have $\tau=\K-\D$ as
\[\tau_k\left(\begin{array}{cc}
a & m\\
n & b
\end{array}\right)=\left(\begin{array}{cc}
(P_k-\mathsf P_k)(a)+P''_k(b) &  \\
 & Q''_k(a)+(Q_k-\mathsf Q_k)(b)
\end{array}\right).\]
We set $\ell_k=P_k-\mathsf P_k, \ell'_k=Q_k-\mathsf Q_k$, one can directly check taht $\{\ell_k\}_{k\in\mathbb{N}}, \{\ell'_k\}_{k\in\mathbb{N}}$ are two sequences of maps satisfying the required properties.
\end{proof}
\begin{remark}\label{faithfulness}
It is worthwhile mentioning that in the case where $M$ is a faithful $(A, B)$-module then;
\begin{enumerate}[\hspace{1em}\rm (i)]
\item In Theorem \ref{LHD}, the conditions $Q''_k([a,a'])=0$ and
$P''_k([b,b']) = 0$ for all $k \in\mathbb N$, are superfluous as those can be acquired from (1), (4) and (5). Indeed, by induction on $k$ we know that for $k=1$ this is true by \cite{Ch0}. Suppose that the result holds for any integer less than $k$. For $a, a'\in A, m\in M,$ from (4) we get
\begin{equation}\label{3.2}
\mathfrak f_{3k}([a,a']m)=\sum_{i+j=k,\, j\neq 0}\big(P_i([a,a'])\mathfrak f_{3j}(m)-\mathfrak f_{3j}(m)Q''_i([a,a'])\big)
\end{equation}
On the other hand, employing  (c) and then (a), we have
\begin{eqnarray}\label{3.3}
\nonumber \mathfrak f_{3k}([a,a']m)&=&\mathfrak f_{3k}(aa'm-a'am)\notag\\
\nonumber &=&\sum_{i+j=k,\, j\neq 0}\big(P_i(a)\mathfrak f_{3j}(a'm)-\mathfrak f_{3j}(a'm)Q''_i(a)\big)\\
\nonumber &&-\big(\sum_{i+j=k,\, j\neq 0}\big(P_i(a')\mathfrak f_{3j}(am)-\mathfrak f_{3j}(am)Q''_i(a')\big)\notag\\
\nonumber &=&\sum_{r+s+t=k}\big(P_r(a)P_s(a^{\prime})\mathfrak f_{3t}(m)-P_r(a)\mathfrak f_{3t}(m)Q''_s(a^{\prime})\big)\\
\nonumber &&-\sum_{k+s+t=k}\big(P_s(a^{\prime})\mathfrak f_{3t}(m)Q''_k(a)-\mathfrak f_{3t}(m)Q''_s(a^{\prime})Q''_r(a)\big)\\
\nonumber &&-\sum_{r+s+t=k}\big(P_r(a^{\prime})P_s(a)\mathfrak f_{3t}(m)-P_r(a^{\prime})\mathfrak f_{3t}(m)Q''_s(a)\big)\\
\nonumber &&+\sum_{r+s+t=k}\big(P_s(a)\mathfrak f_{3t}(m)Q''_r(a^{\prime})-\mathfrak f_{3t}(m)Q''_s(a)Q''_r(a^{\prime})\big)\\
&=&\sum_{r+s+t=k}[P_r(a),P_s(a^{\prime})]\mathfrak f_{3t}(m).
\end{eqnarray}
Comparing the equations \eqref{3.2} and \eqref{3.3} along with assumption of induction indicates that $mQ''_k([a,a'])=0,$ for any $m\in M$, thus the equality $Q''_k([a,a'])=0$ follows from the faithfulness of $M$ (as a right $B-$module). Similarly on can check that $P''_k([b,b'])=0$ for all $k\in \mathbb{N}$.
\item In Theorem \ref{HD}, the assertion $\rm (a)$ can also be removed as it can be acquired from $\rm (d)$ and $\rm (e)$ by a similar reasoning as in (i), (see {\cite[Page 303]{Ch2}}).
\item In Theorem \ref{Asasi}, the same reason as in (ii) indicates that the assertion $\rm (A)$ in Theorem \ref{Asasi}, stating that $P_k-\ell_k$ and $Q_k-\ell_k$ are higher derivations, is extra.
\end{enumerate}
\end{remark}
In the next corollary we offer the criterion characterizing LHD property for the generalized matrix algebra $\mathcal G$ as a conclusion of 
Theorem \ref{Asasi}. 
\begin{corollary}\label{A'}
Let $\mathcal{G}$ be a generalized matrix algebra and $\K$ be a Lie higher derivation on $\mathcal{G}$ of the form stated in Theorem \ref{LHD}. If $\K$ is proper then  
\begin{enumerate}[\hspace{1em}\rm ($A'$)]
\item $Q''_k(A)\subseteq \pi_B(Z(\mathcal{G}))$, $P''_k(B)\subseteq\pi_A(Z(\mathcal{G}))$ and,
\item $P''_k(nm)\oplus Q''_k(mn)\in Z(\mathcal{G}),$ 
\end{enumerate} 
for all $m\in M, n\in N$ and $k\in\mathbb{N}$. The converse is valid when $\mathcal G$ is weakly faithful.
\end{corollary}
%In this case $\ell_A$ and $\ell_B$ are unique; more precisely,  $\ell_\A=\varphi^{-1}\circ h_\A$ and $\ell_\B= \varphi\circ h_\B,$ where $\varphi:\pi_\A(Z(\mathcal{G}))\longrightarrow \pi_\B(Z(\mathcal{G}))$ is  the isomorphism satisfying   $a\oplus\varphi(a)\in Z(\G)$ for all $a\in\A,$ whose existence guaranteed by the faithfulness of $\M.$ So the condition $\rm (B)$ is equivalent to saying that  $h_\A(\A) \subseteq \pi_\B(Z(\mathcal{G}))$ and $h_\B(\B)\subseteq \pi_\A(Z(\mathcal{G})).$
%It should also be remarked that, as we have removed in Theorem \ref{Asasi},  the  equalities $\ell_\A([a,a'])=0$ and $\ell_\B([b,b'])=0$ in {\cite[Theorem 6]{CH1}} are superfluous, however, Cheung has removed them in the case that $\M$ is faithful.
%\end{enumerate}
%\end{remark}
\begin{proof}
Let $\K$ be proper, then Theorem \ref{Asasi} ensures that $\rm (A')$ and $\rm (B')$ hold.
Conversely, suppose that $\mathcal G$ is weakly faithful. Let $\varphi:\pi_A(Z(\mathcal{G}))\longrightarrow \pi_B(Z(\mathcal{G}))$ be the isomorphism satisfying $a\oplus\varphi(a)\in Z(\mathcal G)$ for all $a\in A,$  whose existence guaranteed by the weak faithfulness of $\mathcal G$ and \cite{Be}. By using assumption $\rm (A')$, we define $\ell_k:A\longrightarrow Z(A)$ and
$\ell'_k:B \longrightarrow Z(B)$  by $\ell_k=\varphi^{-1}\circ Q''_k$ and
$\ell'_k= \varphi\circ P''_k$. Obviously $\ell_k(a)\oplus Q''_k(a)\in Z(\mathcal G)$ and  $P''_k(b)\oplus \ell'_k(b)\in Z(\mathcal G),$ for all $a\in A, b\in B.$
%By Theorem \ref {HH}, we need  to show that $p_\A=P-\ell_\A$ and $q_k=Q_k-\ell'_k$ are derivations. From Proposition \ref{F}, for each $a\in \A$ and $m\in \M $ we have
% \[ f(am)= P(a)m - mh_\A(a)+af(m) = (P(a)-\ell_\A(a))m+af(m)=p_\A(a)m+af(m),\]
%so  for each $a,a'\in A$ and $m \in\M$ we have $f(aa'm)=p_\A(aa')m+aa'f(m).$ On the other hand,
%  \[f(aa'm)= p_\A(a)a'm+af(a'm)=p_\A(a)a'm+a big(p_\A(a')m+a'f(m)\big).\]
%Thus  $p_\A(aa')m=(p_\A(a)a'+ap_\A(a'))m$ and the faithfulness of $\M$ as a left $\A-$module implies that  $p_\A$
%is a derivation. Similarly  $q_ B=Q-\ell_ B$ is a derivation.
Further, $\rm (B')$ follows  that  \[\ell_k(mn)=\varphi^{-1}(Q''_k(mn))=P''_k(nm) \ {\rm  and}\ \ell'_ k(nm)=\varphi(P''_k(nm))=Q''_k(mn).\]
Now properness of $\K$ follows from Theorem \ref{Asasi} and part (iii) of Remark \ref{faithfulness}.
\end{proof}
%%%%%%%%%%%%%%%%%%%%%%%%%%%%%%%%%%%%%%%%%%%%%%%%%%%%%%%%%%%%%%%%%%%%%%%
\section{\color{SEC} Some sufficient conditions and the main result}
By using Corollary \ref{A'} in the next theorem we give the ``higher'' version of a modification of Du and Wang's result \cite[Theorem 1]{DW}, (see also \cite[Corollary 1]{WW} and \cite[Theorem 2.1]{W} in the case $n=2$).
\begin{theorem}\label{ideal}
Let $\mathcal{G}$ be a weakly faithful generalized matrix algebra. If
\begin{enumerate}[\hspace{1em}\rm (i)]
\item $\pi_A(Z(\mathcal{G}))=Z(A), \pi_B(Z(\mathcal{G}))=Z(B)$, and
\item either $A$ or $B$ does not contain nonzero central ideals,
\end{enumerate}
then $\mathcal{G}$ has LHD property.
\end{theorem}
\begin{proof}
By Corollary \ref{A'}, it is enough to show that $P''_k(nm)+Q''_k(mn) \in Z(\mathcal{G})$ for all $m\in M, n\in N$. Without loss of generality suppose that $A$ has no nonzero central ideal. Put
\[\gamma_k(a,b)=\ell_k(a)+P''_k(b)\quad (a\in A,\ b\in B,\ k\in\mathbb{N}),\]
where, as in the proof of the above corollary,  $\ell_k=\varphi^{-1}\circ Q''_k$ and that  $\mathsf P_k=P_k-\ell_k$ is a higher derivation. Now equation \eqref{3.4} implies that
\[\mathsf P_k(amn)=\sum_{i+j=k}\mathfrak f_{3i}(am)\mathfrak g_{4j}(n)-\gamma_k(amn,-nam).\]
The latter equation with the fact that $\mathsf P_k$ is a higher derivation follows that
\[\sum_{i+j=k}\mathsf P_i(a)\mathsf P_j(mn)=\sum_{l+r+j=k}\mathsf P_l(a)\mathfrak f_{3r}(m)\mathfrak g_{4j}(n)-\gamma_k(amn,-nam).\]
By using assumption of induction we get
\[a\mathsf P_k(mn)=\sum_{r+j=k}a\mathfrak f_{3r}(m)\mathfrak g_{4j}(n)-\gamma_k(amn,-nam).\]
Multiply equation \eqref{3.4} from the left by $a$, then we have
\[a\mathsf P_k(mn)=a\sum_{i+j=k}\mathfrak f_{3i}(m)\mathfrak g_{4j}(n)-a\gamma_k(mn,-nm),\]
for all $a\in A, m\in M$ and $n\in N$. The two last equations imply that the set $A\gamma_k(mn,-nm)$ is a central ideal of $A$ for each pair of elements $m\in M, n\in N$. Hence $\ell_k(mn)-P''_k(nm)=\gamma_k(mn,-nm)=0$ and so
$P''_k(nm)\oplus Q''_k(mn)=\ell_k(mn)\oplus Q''_k(mn) \in Z(\mathcal{G})$.
\end{proof}
As some examples of an algebra that has no nonzero central ideal we can mention to a noncommutative unital prime algebra with a nontrivial idempotent, in particular B(X), the algebra of operators on a Banach space $X$ with dim$(X)> 1$, and the full matrix matrix algebra $M_n(A)$ with $n\geq 2$ (see \cite[Lemma 1]{DW}). Also in \cite[Theorem  2]{DW} it is shown that in the generalized matrix algebra $\mathcal G$ with loyal $M$, $A$ does not contain central ideal if $A$ is noncommutative.

%\begin{remark}
%Let $A$ be an agebra and $Z(A)$ be its center. If $\D=\{d_k\}_{k\in\mathbb{N}}$ be a higher derivation on $A$, then $Z(A)$ is invariant under $\D$, i.e. $\D$ maps $Z(A)$ into $Z(A)$.
%\end{remark}
%By using above remark, in the next result we provide some new sufficient conditions assuring the LHD property for $\mathcal G$. 

Parallel to the results of \cite{EM} we have the three following theorems which all of them can be proved by induction and using techniques of \cite{EM} for step 1.

Recall that an algebra $A$ is called domain if it has no zero devisors or equivalently if $aa'=0$ implies $a=0$ or $a'=0$ for every two elements $a,a'\in A$.
\begin{theorem}\label{domain}
Let $\mathcal{G}$ be a weakly faithful generalized matrix algebra.
Then  $\mathcal{G}$  has LHD property if
\begin{enumerate}[\hspace{1em}\rm (i)]
\item $\pi_A(Z(\mathcal{G}))=Z(A),  \pi_B(Z(\mathcal{G}))=Z(B)$ and
\item $A$ and $B$ are domain.
\end{enumerate}
\end{theorem}
%\begin{proof}
%By induction and using proof of step1 the result holds. 
%\end{proof}
%As a consequnce of Theorem \ref{domain} we have the next theorem.
\begin{theorem}\label{strong}
The generalized matrix algebra $\mathcal G$ has LHD property if
\begin{enumerate}[\hspace{1em}\rm (i)]
\item $\pi_A(Z(\mathcal{G}))=Z(A), \pi_B(Z(\mathcal{G}))=Z(B)$ and
\item either $M$ or $N$ is strongly faithful.
\end{enumerate}
\end{theorem}

It's remarkable that, we do not know when one can withdraw the assertion strong faithfulness in Theorem \ref{strong}.\\

Now, by gathering the above observations and combination of the assertions in Theorems \ref{ideal}, \ref{domain} and \ref{strong}, we are able to give the main result of this paper providing several sufficient conditions that guarantee the LHD property for a generalized matrix algebra, which one part of its is a generalization of {\cite[Theorem 3.3]{ME}}.
%Motivated by the Cheung's idea \cite{Ch2}, according the latter observations and some suitable combinations of the various assertions in Theorems \ref{ideal}, \ref{domain} and \ref{strong}, we use Theorem \ref{Asasi} and Corollary \ref{A'} to  
\begin{theorem}\label{main}
Let $\mathcal G$ be a weakly faithful generalized matrix algebra. If the following two conditions hold:
\begin{enumerate}[\hspace{1em}\rm (I)]
\item $\pi_A(Z(\mathcal G))=Z(A)$, $\pi_B(Z(\mathcal G))=Z(B)$
\item one of the following conditions holds:
\begin{enumerate}[\hspace{1em}\rm (i)]
\item either $A$ or $B$ does not contain nonzero central ideals
\item $A$ and $B$ are domain
\item either $M$ or $N$ is strongly faithful,
\end{enumerate}
\end{enumerate}
then $\mathcal G$  has LHD property.
\end{theorem}
%\begin{proof}
%It sufficient to say that (I) and (II) imply (A), (B) and (C) of Theorem \ref{Asasi}.  
%\end{proof} 
%%%%%%%%%%%%%%%%%%%%%%%%%%%%%%%%%%%%%%%%%%%%%%%%%%
\section{\color{SEC}Applications}
In this section we investigate LHD property for some main examples of a generalized matrix algebra which includes: trivial generalized matrix algebras, triangular algebras, unital algebras with a nontrivial idempotent, the algebra $B(X)$ of operators on a Banach space $X$ and the full matrix algebra $M_n(A)$ on a unital algebra $A$. 
\subsection*{\color{SEC}LHD property of trivial generalized matrix algebras and  ${\rm Tri}(A,M,B)$}
The generalized matrix algebra $\mathcal G$ is called trivial when $MN=0$ and $NM=0$ in its definition.  
It can be pointed out to the triangular algebra ${\rm Tri}(A,M,B)$ as a main example of a trivial generalized matrix algebra that whose LHD property has been studied in \cite{MO, ME, Q, WX, XW1}. As a urgent consequence of Corollary \ref{A'} and Theorem \ref{main} we achieve the next result which characterizing the LHD property for trivial generalized matrix algebras.
\begin{corollary}\label{trivial}
Let $\mathcal{G}$ be a trivial generalized matrix algebra and $\K$ be a Lie higher derivation on $\mathcal{G}$ of the form stated in ($\bigstar$). If $\K$ is proper then $Q''_k(A)\subseteq \pi_B(Z(\mathcal{G}))$, $P''_k(B)\subseteq\pi_A(Z(\mathcal{G}))$.
The converse is valid when $\mathcal G$ is weakly faithful.

Specifically, a trivial generalized matrix algebra $\mathcal G$ has LHD property if the following two conditions hold:
\begin{enumerate}[\hspace{1em}\rm (I)]
\item $\mathcal G$ is weakly faithful, 
\item $\pi_A(Z(\mathcal G))=Z(A)$ and $\pi_B(Z(\mathcal G))=Z(B)$.
\end{enumerate}
\end{corollary}
In the next example which has been raised by Benkovi\v{c} \cite[Example 3.8]{Be} and modified in \cite{EM} we give a trivial generalized matrix algebra, which is not triangular, without the LHD property.
\begin{example}
Let $M$ be a commutative unital algebra of dimension $3$, on the commutative unital ring $R$, with base $\{1, m, m'\}$ such that ${m}^2={m'}^2=mm'=m'm=0$. Put $N=M$ and let $A=\{r+r'm|\ r, r'\in R\}$ and $B=\{u+u'm'|\ u, u'\in R\}$ be the subalgebras of $M$. Consider the generalized matrix algebra $\mathcal G=\left(\begin{array}{cc}
A & M\\
N & B\\
\end{array}\right)$ under the usual addition, usual scalar mulitiplication and the multiplication defined by
\[\left(\begin{array}{cc}
a & m\\
n& b\\
\end{array}\right)\left(\begin{array}{cc}
a' & m'\\
n'& b'\\
\end{array}\right)=\left(\begin{array}{cc}
aa' & am'+mb'\\
na'+bn'& bb'\\
\end{array}\right).\]
The generalized matrix algebra $\mathcal G$ is trivial since $MN=0=NM$. The linear map $\K:\mathcal G\longrightarrow\mathcal G$ defined by \begin{eqnarray*}
\K\left(\begin{array}{cc}
r+r'm & s+s'm+s''m'\\
t+t'm+t''m'& u+u'm'\\
\end{array}\right)=\left(\begin{array}{cc}
 u'm& -s''m-s'm'\\
-t''m-t'm'& r'm'\\
\end{array}\right),
\end{eqnarray*}
where all coefficients are in the ring $R$, is an improper Lie derivation. Now the ordinary Lie higher derivation induced by $\K$ is improper.
\end{example}
Applying Theorems \ref{LHD} and \ref{HD} for the special case $N=0$ we arrive to the following characterizations of (Lie) higher derivations for the triangular algebra ${\rm Tri}(A, M, B)$ which have already presented in \cite{ME}.
\begin{corollary}
$\bullet$ Let $\K=\{\K_k\}_{k\in\mathbb{N}}$ be a sequence of linear maps on ${\rm Tri}(A, M, B)$, then
$\K$ is a Lie higher derivation if and only if $\K_k$ can be presented in the form\\
\small \[\K_k\left(\begin{array}{cc}
 a & m \\
 & b \\
\end{array}\right)=\left(\begin{array}{cc}
\mathfrak p_{1k}(a)+\mathfrak p_{2k}(b) & \sum_{i+j=k, i\neq k}\big((\mathfrak p_{1i}(a)+\mathfrak p_{2i}(b))m_j-m_j(\mathfrak q_{2i}(b)+\mathfrak q_{1i}(a))\big)+\mathfrak f_{3k}(m) \\
 & \mathfrak q_{1k}(a)+\mathfrak q_{2k}(b)\\
\end{array}\right)\]
where $\{m_j\}_{j\in\Bbb{N}}\subseteq M$, and  for each  $k\in\Bbb{N}$, $\mathfrak q_{1k}:A\longrightarrow Z(B)$, $\mathfrak p_{2k}:B\longrightarrow Z(A),$ $\mathfrak f_{3k}:M\longrightarrow M$ are linear maps satisfying:
\begin{enumerate}[\hspace{1em}\rm (1)]
\item $\{\mathfrak p_{1k}\}_{n\in\Bbb{N}}$, $\{\mathfrak q_{2k}\}_{n\in\Bbb{N}}$ are Lie higher derivations on $A, B$, respectively,

\item $\mathfrak q_{1k}[a,a^{\prime}]=0$ and $\mathfrak p_{2k}[b,b^{\prime}]=0$ for all $a,a'\in A, b,b'\in B,$ and

\item $\mathfrak f_{3k}(am)=\sum_{i+j=k}\big(\mathfrak p_{1i}(a)\mathfrak f_{3j}(m)-\mathfrak f_{3j}(m)\mathfrak q_{1i}(a)\big),$ \\
$\mathfrak f_{3k}(mb)=\sum_{i+j=k}\big(\mathfrak f_{3j}(m)\mathfrak q_{2i}(b)-\mathfrak p_{2i}(b))\mathfrak f_{3j}(m)\big)$ for all $a\in A,b\in B, m\in M.$
\end{enumerate}
$\bullet$ Let $\D=\{\D_k\}_{k\in\mathbb{N}}$ be a sequence of linear maps on ${\rm Tri}(A, M, B)$, then
$\D$ is a higher derivation if and only if, $\D_k$ can be presented in the form \\
\[\small\D_k\left(\begin{array}{cc}
 a & m \\
 & b \\
\end{array}\right)=\left(\begin{array}{cc}
\mathtt p_{1k}(a)+\mathtt p_{2k}(b) & \sum_{i+j=k, i\neq k}\big((\mathtt p_{1i}(a)+\mathtt p_{2i}(b))m_j-m_j(\mathtt q_{2i}(b)+\mathtt q_{1i}(a))\big)+\mathtt f_{3k}(m) \\
 & \mathtt q_{1k}(a)+\mathtt q_{2k}(b)\\
\end{array}\right)\]
where $\{m_j\}_{j\in\Bbb{N}}\subseteq M$, and  for each  $k\in\Bbb{N}$, $\mathtt q_{1k}:A\longrightarrow Z(B)$, $\mathtt p_{2k}:B\longrightarrow Z(A),$ $\mathtt f_{3k}:M\longrightarrow M$ are linear maps satisfying:
\begin{enumerate}[\hspace{1em}\rm (1)]
\item $\{\mathtt p_{1k}\}_{n\in\Bbb{N}}$, $\{\mathtt q_{2k}\}_{n\in\Bbb{N}}$ are  Lie higher derivations on $A, B$, respectively,

\item $\mathtt q_{1k}[a,a^{\prime}]=0$ and $\mathtt p_{2k}[b,b^{\prime}]=0$ for all $a,a'\in A, b,b'\in B,$ and

\item $\mathtt f_{3k}(am)=\sum_{i+j=k}\big(\mathtt p_{1i}(a)\mathtt f_{3j}(m)-\mathtt f_{3j}(m)\mathtt q_{1i}(a)\big),$ \\
$\mathtt f_{3k}(mb)=\sum_{i+j=n}\big(\mathtt f_{3j}(m)\mathtt q_{2i}(b)-\mathtt p_{2i}(b))\mathtt f_{3j}(m)\big)$ for all $a\in A,b\in B, m\in M.$
\end{enumerate}
\end{corollary}

\subsection*{\color{SEC}LHD property of unital algebras with a nontrivial idempotent}
Let $A$ be a unital algebra with a nontrivial idempotent $e$ and $f=1-e$. From the Peirce decomposition we can presented $A$ as {\small $A=\left(\begin{array}{cc}
eAe & eAf\\
fAe & fAf\\
\end{array}\right)$}. By using Theorem \ref{main} for the generalized matrix algebra $A$ we get the next result which is the ``higehr" version of \cite[Corollary 4.3]{EM}.
\begin{corollary}\label{idempotent}
Let $A$ be a $2-$torsion free unital algebra with a nontrivial idempotent $e$  satisfying
\begin{equation}\label{faithful}
eae\cdot eAf=0 \ implies \ eae=0,\quad and\quad
eAf\cdot faf=0 \ implies \ faf=0,
\end{equation}
for any $a\in A,$ where  $f=1-e$. If the following conditions hold:
\begin{enumerate}[\hspace{1em}\rm (I)]
\item {\small $Z(fAf)=Z(A)f$, $Z(eAe)=Z(A)e$}
%\item {\small $Z(pAp)=Z(A)p$, or $qAq=\mathcal{W}_{qAq}$, or $qAq$ has LHD property and $qAq=\mathcal{W}_{qAq}.$}
\item one of the following three conditions holds:
\begin{enumerate}[\hspace{1mm}\rm (i)]
\item either $eAe$ or $fAf$ does not contain nonzero central ideals
\item $eAe$ and $fAf$ are domain
\item either $eAf$ or $fAe$ is strongly faithful,
\end{enumerate}
\end{enumerate}
then $A$ has LHD property.
\end{corollary}
%\subsection*{\color{SEC}Lie derivations on $B(X)$}
As urgent consequences of Corollary \ref{idempotent} in the next results we obtain LHD property of the full matrix algebra $M_n(A)$ and $B(X)$, the algebra af all operator on Banach space $X$ with $dim(X)\geq 2$. The LHD property of $B(X)$ with $dim(X)>1$ was proved by Han \cite[Corollary 3.3]{H} by a completely different method. Also the Lie derivation property of $B(X)$ was proved by Lu and Jing \cite{LJ} for Lie derivable maps at zero and idempotents. In addition, for properness of nonlinear Lie derivations on $B(X)$ see \cite{LL}.
\begin{corollary}\label{K8}
The algebra $B(X)$ of bounded operators on a Banach space $X$ with $dim(X)\geq 2$ has LHD property.
\end{corollary}
\begin{proof}
It follows from Corollary \ref{idempotent} and the proof appeared in \cite[Corollary 4.4]{EM}.
%Set $A=B(X)$. Consider a nonzero element $x_0 \in X$ and $f_0 \in X^*$ such that $f_0(x_0)=1,$  then $p=x_0\otimes f_0$
%defined by $y \longmapsto f_0(y)x_0$ is a nontrivial idempotent. A direct verification reveals that $A$ satisfies the implications \eqref{faithful}. Indeed, if $pTp\cdot pAq=\{0\}$ for some $T\in A$, then choose a nonzero element  $y\in q(X)$ such that $q(y)=y.$ Let $x\in X$ then there exists an operator   $S\in B(X)$ such that $S(y)=x$ (e.g. $S:=x\otimes g$ for some $g\in X^*$ with $g(y)=1$). We then get $pTp(x)=pTp(S(q(y)))=pTp\cdot pSq(y)=0.$
%Further, it can be readily verified that   $Z(A)=\mathbb{C}I_X=\mathbb{C}(p+q)$, $Z(pAp)=\mathbb{C}p$ and $Z(qAq)=\mathbb{C}q.$ In particular, $Z(pAp)=Z(A)p$, $Z(qAq)=Z(A)q.$
%These also imply that neither $pAp$ nor $qAq$ has no central ideals. By Corollary \ref{idempotent} $A=B(X)$ has LHD property.
\end{proof}
%\subsection*{\color{SEC}Lie derivations on $M_n(\A)$}
\begin{corollary}\label{fullmatrix}
Let $A$ be a $2-$torsion free unital algebra. The full matrix algebra $\mathfrak A=M_n(A)$ with $n\geq 3$ enjoys the LHD property.
\end{corollary}
\begin{proof}
Consider nontrivial idempotents $e=e_{11}$ and $f=e_{22}+\cdots +e_{nn}$. It is obvious that $e\mathfrak Ae=A, f\mathfrak Af=M_{n-1}(A)$. From $Z(\mathfrak A)=Z(A)1_{\mathfrak A}$ we conclude that $Z(e\mathfrak Ae)=Z(\mathfrak A)e$ and $Z(f\mathfrak A f)=Z(\mathfrak A)f,$ so assumption (I) of Corollary \ref{idempotent} holds. Moreover, \cite[Lemma 1]{DW} guarantees that the algebra $f\mathfrak Af=M_{n-1}(A)$ does not contain nonzero central ideals, so part (i) of condition (II) in Corollary \ref{idempotent} is fulfilled. Hence by the mentioned corollary $M_n(A)$ has the LHD property.
\end{proof}
It's remarkable that Corollary \ref{fullmatrix} is the ``higher" version of \cite[Corollary 1]{DW}. 

\section{\color{SEC}Proofs of Theorems \ref{LHD} and \ref{HD}}
\subsection*{\color{SEC}Proof of Theorem \ref{LHD}}
\begin{proof}
We proceed the proof by induction on $k$. The case $k=1$ follows  from Proposition \ref{k=1}. Suppose that the conclusion holds for any integer less than $k$. By ($\bigstar$), $\K_k$ has the presentation
\[\K_k\left(\begin{array}{cc}
a & m\\
n & b
\end{array}\right)=\left(\begin{array}{cc}
\mathfrak p_{1k}(a)+\mathfrak p_{2k}(b)+\mathfrak p_{3k}(m)+\mathfrak p_{4k}(n) & \mathfrak f_{1k}(a)+\mathfrak f_{2k}(b)+\mathfrak f_{3k}(m)+\mathfrak f_{4k}(n) \\
\mathfrak g_{1k}(a)+\mathfrak g_{2k}(b)+\mathfrak g_{3k}(m)+\mathfrak g_{4k}(n) & \mathfrak q_{1k}(a)+\mathfrak q_{2k}(b)+\mathfrak q_{3k}(m)+\mathfrak q_{4k}(n) \\
\end{array}\right),\]
for each $\left(\begin{array}{cc}
a & m\\
n & b
\end{array}\right)\in\mathcal G.$ Applying  $\K_k$ for the  commutator $\left[\left(\begin{array}{cc}
a & m\\
0 & 0
\end{array}\right),\left(\begin{array}{cc}
0 & 0\\
0 & b
\end{array}\right)\right]$, we have
%\vspace*{-1.1cm}
%\begin{eqnarray}
%\label{Liematrix}
%\nonumber &&\left(\begin{array}{cc}
%\mathfrak p_{3k}(mb) & \mathfrak f_{3k}(mb)\\
%\mathfrak g_{3k}(mb) & \mathfrak q_{3k}(mb)
%\end{array}\right)\\
%\nonumber &=&\sum_{i+j=k, i, j\neq 0}\left[\left(\begin{array}{cc}
%\mathfrak p_{1i}(a)+\mathfrak p_{3i}(m) & \mathfrak f_{1i}(a)+ \mathfrak f_{3i}(m)\\
%\mathfrak g_{1i}(a)+\mathfrak g_{3i}(m) & \mathfrak q_{1i}(a)+ \mathfrak q_{3i}(m)
%\end{array}\right),\left(\begin{array}{cc}
%\mathfrak p_{2j}(b) & \mathfrak f_{2j}(b)\\
%\mathfrak g_{2j}(b) & \mathfrak q_{2j}(b)
%\end{array}\right)\right]\\
%\nonumber &&+\left[\left(\begin{array}{cc}
%\mathfrak p_{1k}(a)+\mathfrak p_{3k}(m) & \mathfrak f_{1k}(a)+ \mathfrak f_{3k}(m)\\
%\mathfrak g_{1k}(a)+\mathfrak g_{3k}(m) & \mathfrak q_{1k}(a)+\mathfrak q_{3k}(m)
%\end{array}\right),\left(\begin{array}{cc}
%0 & 0\\
%0 & b
%\end{array}\right)\right]\\
%&&+\left[\left(\begin{array}{cc}
%a & m\\
%0 & 0
%\end{array}\right),\left(\begin{array}{cc}
%\mathfrak p_{2k}(b) & \mathfrak f_{2k}(b)\\
%\mathfrak g_{2k}(b) & \mathfrak q_{2k}(b)
%\end{array}\right)\right].
%\end{eqnarray}
\begin{align}
\label{Liematrix}
&\begin{pmatrix}
\mathfrak p_{3k}(mb) & \mathfrak f_{3k}(mb)\\
\mathfrak g_{3k}(mb) & \mathfrak q_{3k}(mb)
\end{pmatrix}\cr
 &\quad =\sum_{i+j=k, i, j\neq 0}\left[\begin{pmatrix}
\mathfrak p_{1i}(a)+\mathfrak p_{3i}(m) & \mathfrak f_{1i}(a)+ \mathfrak f_{3i}(m)\\
\mathfrak g_{1i}(a)+\mathfrak g_{3i}(m) & \mathfrak q_{1i}(a)+ \mathfrak q_{3i}(m)
\end{pmatrix},\begin{pmatrix}
\mathfrak p_{2j}(b) & \mathfrak f_{2j}(b)\\
\mathfrak g_{2j}(b) & \mathfrak q_{2j}(b)
\end{pmatrix}\right]\cr
&\qquad +\left[\begin{pmatrix}
\mathfrak p_{1k}(a)+\mathfrak p_{3k}(m) & \mathfrak f_{1k}(a)+ \mathfrak f_{3k}(m)\\
\mathfrak g_{1k}(a)+\mathfrak g_{3k}(m) & \mathfrak q_{1k}(a)+\mathfrak q_{3k}(m)
\end{pmatrix},\begin{pmatrix}
0 & 0\\
0 & b
\end{pmatrix}\right]\cr
&\qquad +\left[\begin{pmatrix}
a & m\\
0 & 0
\end{pmatrix},\begin{pmatrix}
\mathfrak p_{2k}(b) & \mathfrak f_{2k}(b)\\
\mathfrak g_{2k}(b) & \mathfrak q_{2k}(b)
\end{pmatrix}\right].
\end{align}
%\vspace*{-.1cm}
Use the equalities  $\mathfrak p_{1i}=P_i-p_i, \mathfrak q_{2j}=Q_j-q_j, \mathfrak q_{1i}=Q''_i+q''_i, \mathfrak p_{2j}=P''_j+p''_j$, it follows that,
%\vspace*{-2.2cm}
\begin{eqnarray}\label{f_{3k}(mb)}
\nonumber \mathfrak f_{3k}(mb)&=&\mathfrak f_{1k}(a)b+\mathfrak f_{3k}(m)b+a\mathfrak f_{2k}(b)+m\mathfrak q_{2k}(b)-\mathfrak p_{2k}(b)m\\
\nonumber &&+\sum_{i+j=k, i, j\neq 0}\big(\big(\mathfrak p_{1i}(a)+\mathfrak p_{3i}(m)\big)\mathfrak f_{2j}(b)+\big(\mathfrak f_{1i}(a)+\mathfrak f_{3i}(m)\big)\mathfrak q_{2j}(b)\big)\\
&&-\sum_{i+j=k, i, j\neq 0}\big(\mathfrak p_{2j}(b)\big(\mathfrak f_{1i}(a)+\mathfrak f_{3i}(m)\big)+\mathfrak f_{2j}(b)\big(\mathfrak q_{1i}(a)+\mathfrak q_{3i}(m)\big)\big)\\
\nonumber &=&\mathfrak f_{1k}(a)b+\mathfrak f_{3k}(m)b+a\mathfrak f_{2k}(b)+m\mathfrak q_{2k}(b)-\mathfrak p_{2k}(b)m\\
\nonumber &&+\sum_{i+j=k, i, j\neq 0}\big(P_i(a)\mathfrak f_{2j}(b)-\mathfrak f_{2j}(b)Q''_i(a)\big)-\sum_{i+j=k, i, j\neq 0}p_i(a)\mathfrak f_{2j}(b)\\
\nonumber &&+\sum_{i+j=k, i, j\neq 0}\big(\mathfrak f_{1i}(a)Q_j(b)-P''_j(b)\mathfrak f_{1i}(a)\big)+\sum_{i+j=k, i, j\neq 0}\mathfrak p_{3i}(m)\mathfrak f_{2j}(b)\\
\nonumber &&+\sum_{i+j=k, i, j\neq 0}\big(\mathfrak f_{3i}(m)Q_j(b)- P''_j(b)\mathfrak f_{3i}(m)\big)-\sum_{i+j=k, i, j\neq 0}\mathfrak f_{1i}(a)q_j(b)\\
\nonumber &&-\sum_{i+j=k, i, j\neq 0}\mathfrak f_{3i}(m)q_j(b)-\sum_{i+j=k, i, j\neq 0}p''_j(b)\mathfrak f_{3i}(m)-\sum_{i+j=k, i, j\neq 0}p''_j(b)\mathfrak f_{1i}(a)\\
\nonumber &&-\sum_{i+j=k, i, j\neq 0}\mathfrak f_{2j}(b)q''_i(a)-\sum_{i+j=k, i, j\neq 0}\mathfrak f_{2j}(b)\mathfrak q_{3i}(m)\label{f'_{3k}(mb)}\\
\end{eqnarray}
In  \eqref{f'_{3k}(mb)}, if we  put $m=0, a=1, b=1$ and use the definition of $p_i, p''_i, q_i$ and $q''_i$ appeared in page $5$ for $a=1, b=1$ then we arrive at, 
\begin{eqnarray*}
0&=&\mathfrak f_{1k}(1)+\mathfrak f_{2k}(1)-\sum_{i+j=k, i, j\neq 0}p_i(1)\mathfrak f_{2j}(1)-\sum_{i+j=k, i, j\neq 0}\mathfrak f_{2j}(1)q''_i(1)\\
&&-\sum_{i+j=k, i, j\neq 0}\mathfrak f_{1i}(1)q_j(1)-\sum_{i+j=k, i, j\neq 0}p''_j(1)\mathfrak f_{1i}(1)\\
&=&\mathfrak f_{1k}(1)+\mathfrak f_{2k}(1)+\sum_{i+j=k, i, j\neq 0}\sum_{r=1}^{\eta_i}\sum_{(\alpha+\beta)_r=i}m_{\beta_1}n_{\alpha_1}\ldots m_{\beta_r}n_{\alpha_r}m_{j}\\
&&+\sum_{i+j=k, i, j\neq 0}\sum_{r=1}^{\eta_i}\sum_{(\alpha+\beta)_r=i}m_{j}n_{\alpha_r}m_{\beta_r}\ldots n_{\alpha_1}m_{\beta_1}\\
&&-\sum_{i+j=k, i, j\neq 0}\sum_{r=1}^{\eta_j}\sum_{(\alpha+\beta)_r=j}m_{i}n_{\alpha_r}m_{\beta_r}\ldots n_{\alpha_1}m_{\beta_1}\\
&&-\sum_{i+j=k, i, j\neq 0}\sum_{r=1}^{\eta_j}\sum_{(\alpha+\beta)_r=j}m_{\beta_1}n_{\alpha_1}\ldots m_{\beta_r}n_{\alpha_r}m_{i};
\end{eqnarray*}
from which we get $\mathfrak f_{1k}(1)=-\mathfrak f_{2k}(1)$. Note that in the recent calculations, by induction hypothesis we have $\mathfrak f_{1j}(1)=-\mathfrak f_{2j}(1)=m_j$ for all $j<k$ and 
\[\sum_{i+j=k, i, j\neq 0}\big(P_i(1)m-mQ''_i(1)\big)=0,\sum_{i+j=k, i, j\neq 0}\big(mQ_j(1)-P''_j(1)m\big)=0\]
for all $m\in M$.

Next, if we apply \eqref{f_{3k}(mb)} for $m=0,\, b=1$ and use the equations $ \mathfrak q_{2j}=Q_j-q_j$, $\mathfrak p_{2j}=P''_j+p''_j$ then we have
\begin{eqnarray*} 
\mathfrak f_{1k}(a)&=&a\mathfrak f_{1k}(1)+\sum_{i+j=k, i, j\neq 0}\big(\mathfrak p_{1i}(a)\mathfrak f_{1j}(1)-\mathfrak f_{1i}(a)\mathfrak q_{2j}(1)+\mathfrak p_{2j}(1)\mathfrak f_{1i}(a)-\mathfrak f_{1j}(1)\mathfrak q_{1i}(a) \big)\\
&=&\sum_{i+j=k, j\neq 0}\big(\mathfrak p_{1i}(a)m_j-m_j\mathfrak q_{1i}(a) \big)+\sum_{i+j=k, i, j\neq 0}\big(\mathfrak p''_j(1)\mathfrak f_{1i}(a)+\mathfrak f_{1i}(a)\mathfrak q_j(1)\big)\\
&=&\sum_{i+j=k, j\neq 0}\big(\mathfrak p_{1i}(a)m_j-m_j\mathfrak q_{1i}(a) \big)\\
&&+\sum_{i+j=k, i, j\neq 0}\big(\sum_{r=1}^{\eta_j}\sum_{(\alpha+\beta)_r=j}m_{\beta_1}n_{\alpha_1}\ldots m_{\beta_r}n_{\alpha_r}\sum_{s+t=i}(P_t(a)m_s-m_sQ''_t(a))\big)\\
&&+\sum_{i+j=k, i, j\neq 0}\big(\sum_{s+t=i}\big(P_t(a)m_s-m_sQ''_t(a)\big) \sum_{r=1}^{\eta_j}\sum_{(\alpha+\beta)_r=j}n_{\alpha_r}m_{\beta_r}\ldots n_{\alpha_1}m_{\beta_1}\big)\\
&=&\sum_{i+j=k, j\neq 0}\big(P_i(a)m_j-m_jQ''_i(a)\big).
\end{eqnarray*}
It's remarkable that $\sum_{i+j=k, i, j\neq 0}\big(P''_j(1)\mathfrak f_{1i}(a)-\mathfrak f_{1i}(a)Q_j(1)\big)=0$ by induction hypothesis.\\
%\begin{eqnarray*}
%\mathfrak f_{1k}(a)&=&\sum_{i+j=k,  j\neq 0}\big(P_i(a)m_j-m_jQ''_i(a)\big)\\
%&&+\sum_{i+j=k, i, j\neq 0}\big(\mathfrak p_{1i}(a)\mathfrak f_{1j}(1)-\mathfrak f_{1i}(a)\mathfrak q_{2j}(1)+\mathfrak p_{2j}(1)\mathfrak f_{1i}(a)-\mathfrak f_{1j}(1)\mathfrak q_{1i}(a) \big)\\
%&=&\sum_{i+j=k,  j\neq 0}\big(\mathfrak p_{1i}(a)m_j-m_j\mathfrak q_{1i}(a) \big)+\sum_{i+j=k, i, j\neq 0}\big(\mathfrak p''_j(1)\mathfrak f_{1i}(a)+\mathfrak f_{1i}(a)\mathfrak q_j(1)\big)\\
%&=&\sum_{i+j=k,  j\neq 0}\big(\mathfrak p_{1i}(a)m_j-m_j\mathfrak q_{1i}(a) \big)\\
%&&+\sum_{i+j=k, i, j\neq 0}\big(\sum_{r=1}^{\eta_j}\sum_{(\alpha+\beta)_r=j}m_{\beta_1}n_{\alpha_1}\ldots m_{\beta_r}n_{\alpha_r}\sum_{s+t=i}(P_t(a)m_s-m_sQ''_t(a))\big)\\
%&&+\sum_{i+j=k, i, j\neq 0}\big(\sum_{s+t=i}\big(P_t(a)m_s-m_sQ''_t(a)\big) \sum_{r=1}^{\eta_j}\sum_{(\alpha+\beta)_r=j}n_{\alpha_r}m_{\beta_r}\ldots n_{\alpha_1}m_{\beta_1}\big)\\
%&=&\sum_{i+j=k,  j\neq 0}\big(P_i(a)m_j-m_jQ''_i(a)\big).
%\end{eqnarray*}
Similarly we can show that $\mathfrak g_{1k}(1)=-\mathfrak g_{2k}(1)$ and similar equations hold for
$\mathfrak g_{1k}(a), \mathfrak g_{2k}(b), \mathfrak f_{2k}(b)$.
On the other hand when we set $a=0$ in \eqref{f_{3k}(mb)} we get
\small \begin{eqnarray}\label{f_{3k}}
\nonumber \mathfrak f_{3k}(mb)&=&\mathfrak f_{3k}(m)b+m\mathfrak q_{2k}(b)-\mathfrak p_{2k}(b)m+\hspace*{-.4cm}\sum_{i+j=k, i, j\neq 0}\big(\mathfrak p_{3i}(m)\mathfrak f_{2j}(b)+\mathfrak f_{3i}(m)\mathfrak q_{2j}(b)\big)\\
\nonumber &&-\hspace*{-.4cm}\sum_{i+j=k, i, j\neq 0}\big(\mathfrak p_{2j}(b)\mathfrak f_{3i}(m)+\mathfrak f_{2j}(b)\mathfrak q_{3i}(m)\big)\\
&=&\hspace*{-.4cm}\sum_{i+j=k}\big(\mathfrak f_{3i}(m)\mathfrak q_{2j}(b)-\mathfrak p_{2j}(b)\mathfrak f_{3i}(m)\big)+\hspace*{-.4cm}\sum_{i+j=k, i, j\neq 0}\big(\mathfrak p_{3i}(m)\mathfrak f_{2j}(b)-\mathfrak f_{2j}(b)\mathfrak q_{3i}(m)\big).
\end{eqnarray}
We calculated the phrase appeared in the last Sigma of \eqref{f_{3k}}. 
\small{\begin{eqnarray*}
&&\hspace*{-.4cm}\mathfrak p_{3i}(m)\mathfrak f_{2j}(b)\hspace*{-.1cm}-\hspace*{-.1cm}\mathfrak f_{2j}(b)\mathfrak q_{3i}(m)\hspace*{-.3cm}\\
&=&\quad\hspace*{-.4cm}\sum_{s+t=i}\sum_{\lambda+\mu=j}\mathfrak f_{3t}(m)\NN_s(m_{\lambda}Q_{\mu}(b)-P''_{\mu}(b)m_{\lambda})+\sum_{s+t=i}\sum_{\lambda+\mu=j}(m_{\lambda}Q_{\mu}(b)-P''_{\mu}(b)m_{\lambda})\NN_s \mathfrak f_{3t}(m)\\
&=&\mathfrak f_{3i}(m)\sum_{s+\lambda+\mu=j}\NN_s(m_{\lambda}Q_{\mu}(b)-P''_{\mu}(b)m_{\lambda})
+\sum_{s+\lambda+\mu=j}(m_{\lambda}Q_{\mu}(b)-P''_{\mu}(b)m_{\lambda})\NN_s\mathfrak f_{3i}(m)\\
&=&\mathfrak f_{3i}(m)\sum_{s+\lambda+\mu=j}\sum_{r=1}^{\nu_s}\hspace*{-.1cm}\sum_{(\alpha+\beta)_r+\gamma=s}\hspace*{-.5cm}
(n_{\alpha_1}m_{\beta_1}\ldots n_{\alpha_r}m_{\beta_r}n_{\gamma}+n_s)(m_{\lambda}Q_{\mu}(b)-P''_{\mu}(b)m_{\lambda})\\
&&+\sum_{s+\lambda+\mu=j}\sum_{r=1}^{\nu_s}\hspace*{-.1cm}
\sum_{(\alpha+\beta)_r+\gamma=s}\hspace*{-.5cm}
(m_{\lambda}Q_{\mu}(b)-P''_{\mu}(b)m_{\lambda})(n_{\alpha_1}m_{\beta_1}\ldots n_{\alpha_r}m_{\beta_r}n_{\gamma}+n_s)\mathfrak f_{3i}(m)\\
&=&\mathfrak f_{3i}(m)\sum_{r'=1}^{\eta_j}\sum_{(\alpha+\beta)_{r'}+\mu=j,\, \mu\leq j-2}n_{\alpha_{r'}}m_{\beta_{r'}}\ldots n_{\alpha_1}(m_{\beta_1}Q_{\mu}(b)-P''_{\mu}(b)m_{\beta_1})\\
&&+\sum_{r'=1}^{\eta_j}\sum_{(\alpha+\beta)_{r'}+\mu=j,\, \mu\leq j-2}(m_{\beta_1}Q_{\mu}(b)-P''_{\mu}(b)m_{\beta_1})n_{\alpha_1}\ldots m_{\beta_{r'}}n_{\alpha_{r'}}\mathfrak f_{3i}(m).
\end{eqnarray*}}
The indices $s, t$ appeared in the first equation of above relations acquire all values between $0$ to $i$ and $i$ takes all values between $1$ to $k-1$, because of all these changes is symmetric the second equation of above relations hold.\\
Note that one of the maximum length for $(n_{\alpha_1}m_{\beta_1})\ldots (n_{\alpha_r}m_{\beta_r})$ is at
$\mu=0, \lambda=1$, in this case $j=s+1$ and the length of
$(n_{\alpha_1}m_{\beta_1})\ldots (n_{\alpha_{\nu_s}}m_{\beta_{\nu_s}})(n_{\gamma}m_1)$  is
\[\nu_s+1=\nu_{j-1}+1=\left\lbrace\begin{array}{ccc}
j/2 & ; & j-1\in\mathbb{O}\\
(j-1)/2 & ; & j-1\in\mathbb{E}
\end{array}\right.=\left\lbrace\begin{array}{ccc}
j/2 & ; & j\in\mathbb{E}\\
(j-1)/2 & ; & j\in\mathbb{O}
\end{array}\right.=\eta_j.\]
It is remarkable that the length of
\[(n_{\alpha_1}m_{\beta_1})\ldots (n_{\alpha_{\nu_s}}m_{\beta_{\nu_s}})(n_{\gamma}m_{\lambda}),\qquad \sum_{i=1}^{\nu_s}(\alpha_i+\beta_i)+\gamma+\lambda=j,\]
in the case where $\lambda\neq 1$ is the same as length of
$(n_{\alpha_1}m_{\beta_1})\ldots (n_{\alpha_{\nu_s}}m_{\beta_{\nu_s}})(n_{\gamma}m_1)$, or equivalently it is $\eta_j$.
Now by replacing this relation in \eqref{f_{3k}} we get
\[\mathfrak f_{3k}(mb)=\sum_{i+j=k}\big(\mathfrak f_{3i}(m)Q_j(b)-P''_j(b)\mathfrak f_{3i}(m)\big).\]
Similarly one can check similar equations for $\mathfrak f_{3k}(am), \mathfrak g_{4k}(na), \mathfrak g_{4k}(bn)$. Again from \eqref{Liematrix} we get, 
\small \begin{eqnarray}\label{p_{3k}(mb)}
\nonumber\mathfrak p_{3k}(mb)\hspace*{-.2cm}&=&\hspace*{-.2cm}a\mathfrak p_{2k}(b)+m\mathfrak g_{2k}(b)-\mathfrak p_{2k}(b)a
+\hspace*{-.4cm}\sum_{i+j=k,\, i, j\neq 0}\hspace*{-.4cm}\big(\big(\mathfrak p_{1i}(a)+\mathfrak p_{3i}(m)\big)\mathfrak p_{2j}(b)+\big(\mathfrak f_{1i}(a)+\mathfrak f_{3i}(m)\big)\mathfrak g_{2j}(b)\big)\\
&&-\hspace*{-.4cm}\sum_{i+j=k,\, i, j\neq 0}\hspace*{-.4cm}\big(\mathfrak p_{2j}(b)\big(\mathfrak p_{1i}(a)+\mathfrak p_{3i}(m)\big)+\mathfrak f_{2j}(b)\big(\mathfrak g_{1i}(a)+\mathfrak g_{3i}(m)\big)\big).
\end{eqnarray}
If we put  $m=0,\, b=1$ in \eqref{p_{3k}(mb)},  then we have
\begin{eqnarray*}
\mathfrak p_{3k}(m)&=&-mn_k-\sum_{i+j=k,\, i, j\neq 0}\bigg(\mathfrak f_{3i}(m)n_j+m_j\sum_{s+l+t=i}\NN_l\mathfrak f_{3t}(m)\NN_s\bigg)\\
&&+\sum_{i+j=k,\, i, j\neq 0}\mathfrak p_{3i}(m)\sum_{r=1}^{\eta_j}\sum_{(\alpha+\beta)_r=j}m_{\beta_1}n_{\alpha_1}\ldots m_{\beta_r}n_{\alpha_r}\\
&&-\sum_{i+j=k,\, i, j\neq 0}\sum_{r=1}^{\eta_j}\sum_{(\alpha+\beta)_r=j}m_{\beta_1}n_{\alpha_1}\ldots m_{\beta_r}n_{\alpha_r}\mathfrak p_{3i}(m)\\
&=&-\sum_{i+j=k}\mathfrak f_{3i}(m)n_j-\sum_{i+j=k,\, i, j\neq 0}\sum_{s+l+t=i}m_j\NN_l\mathfrak f_{3t}(m)\NN_s\\
&&-\sum_{i+j=k,\, i, j\neq 0}\sum_{r=1}^{\eta_j}\sum_{(\alpha+\beta)_r=j}\sum_{s+t=i}\mathfrak f_{3t}(m)\NN_sm_{\beta_1}n_{\alpha_1}\ldots m_{\beta_r}n_{\alpha_r}\\
&&+\sum_{i+j=k,\, i, j\neq 0}\sum_{r=1}^{\eta_j}\sum_{(\alpha+\beta)_r=j}\sum_{s+t=i}m_{\beta_1}n_{\alpha_1}\ldots m_{\beta_r}n_{\alpha_r}\mathfrak f_{3t}(m)\NN_s\\
&=&-\sum_{i+j=k,\, i, j\neq 0}\sum_{r=1}^{\nu_j}\sum_{(\alpha+\beta)_r+\gamma=j}\mathfrak f_{3i}(m)\bigg(n_j+n_{\gamma}m_{\beta_1}n_{\alpha_1}\ldots m_{\beta_r}n_{\alpha_r}\bigg)\\
&&-\sum_{i+j=k,\, i, j\neq 0}\sum_{r=1}^{\eta_j}\sum_{(\alpha+\beta)_r=j}\sum_{s+t=i}m_{\beta_1}n_{\alpha_1}\ldots m_{\beta_r}n_{\alpha_r}\mathfrak f_{3t}(m)\NN_s\\
&&+\sum_{i+j=k,\, i, j\neq 0}\sum_{r=1}^{\eta_j}\sum_{(\alpha+\beta)_r=j}\sum_{s+t=i}m_{\beta_1}n_{\alpha_1}\ldots m_{\beta_r}n_{\alpha_r}\mathfrak f_{3t}(m)\NN_s\\
&=&-\sum_{i+j=k}\mathfrak f_{3i}(m)\NN_j.
\end{eqnarray*}
Also it is not difficult to check similar relations for $\mathfrak q_{3k}(m), \mathfrak p_{4k}(n), \mathfrak q_{4k}(n)$. From $(2,1)$-entry equation \eqref{Liematrix} we have
\begin{eqnarray*}
\mathfrak g_{3k}(mb)&=&-b(\mathfrak g_{1k}(a)+\mathfrak g_{3k}(m))-\mathfrak g_{2k}(b)a\\
&&+\sum_{i+j=k}\big((\mathfrak g_{1i}(a)+\mathfrak g_{3i}(m))\mathfrak p_{2j}(b)+(\mathfrak q_{1i}(a)+\mathfrak q_{3i}(m))\mathfrak g_{2j}(b)\big)\\
&&-\sum_{i+j=k}\big(\mathfrak g_{2j}(b)(\mathfrak p_{1i}(a)+\mathfrak p_{3i}(m))+\mathfrak q_{2j}(b)(\mathfrak g_{1i}(a)+\mathfrak g_{3i}(m))\big).
\end{eqnarray*}
Set $m=0, b=1$ then we have
\begin{eqnarray*}
2\mathfrak g_{3k}(m)&=&\sum_{i+j=k,\, i, j\neq 0}\big(\mathfrak g_{3i}(m)\mathfrak p_{2j}(1)-\mathfrak q_{2j}(1)\mathfrak g_{3i}(m)-\mathfrak q_{3i}(m)n_j+n_j\mathfrak p_{3i}(m)\big)\\
&=&\sum_{i+j=k,\, i, j\neq 0}(\mathfrak g_{3i}(m)p''_j(1)+q_j(1)\mathfrak g_{3i}(m))\\
&&-\sum_{s+t+j=k}(\NN_s \mathfrak f_{3t}(m)n_j+n_j\mathfrak f_{3t}(m)\NN_s)\\
&=&-\sum_{s+t+l+j=k}\sum_{r=1}^{\eta_j}\sum_{(\alpha+\beta)_r=j}\NN_s \mathfrak f_{3t}(m)\NN_l m_{\beta_1}n_{\alpha_1}\ldots m_{\beta_r}n_{\alpha_r}\\
&&-\sum_{s+t+l+j=k}\sum_{r=1}^{\eta_j}\sum_{(\alpha+\beta)_r=j}n_{\alpha_r}m_{\beta_r}\ldots n_{\alpha_1}m_{\beta_1}\NN_s \mathfrak f_{3t}(m)\NN_l\\
&&-\sum_{s+t+j=k}(\NN_s \mathfrak f_{3t}(m)n_j+n_j\mathfrak f_{3t}(m)\NN_s)\\
&=&-\sum_{s+t+j=k}\big(\NN_s \mathfrak f_{3t}(m)(n_j+\sum_{r=1}^{\nu_j}\sum_{(\alpha+\beta)_r+l=j}n_{\alpha_r}m_{\beta_r}\ldots n_{\alpha_1}m_{\beta_1}n_l)\big)\\
&&-\sum_{s+t+j=k}\big((\sum_{r=1}^{\nu_j}\sum_{(\alpha+\beta)_r+l=j}n_{\alpha_r}m_{\beta_r}\ldots n_{\alpha_1}m_{\beta_1}n_l+n_j)\mathfrak f_{3t}(m)\NN_s\big)\\
&=&-2\sum_{s+t+j=k}\NN_s\mathfrak f_{3t}(m)\NN_j.
\end{eqnarray*}
As $A$ is $2$-torsion free we get $\mathfrak g_{3k}(m)=-\sum_{s+t+j=k}\NN_s\mathfrak f_{3t}(m)\NN_j$. By similar argument and $2$-torsion freeness of $B$, it follows that
$\mathfrak f_{4k}(n)=-\sum_{s+t+j=k}\M_s\mathfrak g_{4t}(n)\M_j$.\\
From $(2,2)$-entry of equation \eqref{Liematrix} we have
\begin{eqnarray*}
\mathfrak q_{3k}(mb)&=&(\mathfrak q_{1k}(a)+\mathfrak q_{3k}(m))b-b(\mathfrak q_{1k}(a)+\mathfrak q_{3k}(m))-\mathfrak g_{2k}(m)b\\
&&+\sum_{i+j=k,\, i, j\neq 0}\big((\mathfrak g_{1i}(a)+\mathfrak g_{3i}(m))\mathfrak f_{2j}(b)+(\mathfrak q_{1i}(a)+\mathfrak q_{3i}(m))\mathfrak q_{2j}(b)\big)\\
&&-\sum_{i+j=k,\, i, j\neq 0}\big(\mathfrak g_{2j}(b)(\mathfrak f_{1i}(a)+\mathfrak f_{3i}(m))+\mathfrak q_{2j}(b)(\mathfrak q_{1i}(a)+\mathfrak q_{3i}(m))\big),
\end{eqnarray*}
Set $m = 0$ in the last equation then we get\\
$0=\mathfrak q_{1k}(a)b-b\mathfrak q_{1k}(a)+\sum_{i+j=k,\, i, j\neq 0}\big(\mathfrak g_{1i}(a)\mathfrak f_{2j}(b)+\mathfrak q_{1i}(a)\mathfrak q_{2j}(b)-\mathfrak g_{2j}(b)\mathfrak f_{1i}(a)-\mathfrak q_{2j}(b)\mathfrak q_{1i}(a)\big)$\\
From replacement $\mathfrak q_{1i}, \mathfrak q_{2j}$ in the last equation with $Q''_i+q''_i, Q_j-q_j$ respectively and assumption of induction we have
\begin{eqnarray*}
0&=&\mathfrak q_{1k}(a)b+\sum_{s+t+\mu+\lambda=k}(n_sP'_t(a)-Q''_t(a)n_s)(P''_{\mu}(b)m_{\lambda}-m_{\lambda}Q_{\mu}(b))\\
&&+\sum_{i+j=k,\, i, j\neq 0}(Q''_i(a)+q''_i(a))(Q_j(b)-q_j(b))\\
&&-b\mathfrak q_{1k}(a)-\sum_{s+t+\mu+\lambda=k}(n_sP''_{\mu}(b)-Q'_{\mu}(b)n_s)(P_t(a)m_{\lambda}-m_{\lambda}Q''_t(a))\\
&&-\sum_{i+j=k,\, i, j\neq 0}(Q_j(b)-q_j(b))(Q''_i(a)+q''_i(a))
\end{eqnarray*}
hence
\begin{eqnarray*}
0&=&\mathfrak q_{1k}(a)b-\sum_{s+t+\lambda=k}(n_sP'_t(a)-Q''_t(a)n_s)m_{\lambda}b\\
&&-\sum_{i+j=k,\, i, j\neq 0}q''_i(a)\sum_{r=1}^{\eta_j}\sum_{(\alpha+\beta)_r=j,}n_{\alpha_r}m_{\beta_r}\ldots n_{\alpha_1}m_{\beta_1}b\\
&&-b\mathfrak q_{1k}(a)+\sum_{s+t+\lambda=k}bn_s(P_t(a)m_{\lambda}-m_{\lambda}Q''_t(a))\\
&&+\sum_{i+j=k,\, i, j\neq 0}\sum_{r=1}^{\eta_j}\sum_{(\alpha+\beta)_r=j,}bn_{\alpha_r}m_{\beta_r}\ldots n_{\alpha_1}m_{\beta_1}q''_i(a)\\
&=&\mathfrak q_{1k}(a)b-\sum_{s+t+\lambda=k}(n_sP'_t(a)-Q''_t(a)n_s)m_{\lambda}b\\
&&-\sum_{i+j=k}\sum_{r'=1}^{\eta_i}\sum_{s+(\alpha+\beta)_{r'}=i}
\sum_{r=1}^{\eta_j}\sum_{(\alpha+\beta)_r=j,}(n_{\alpha_1}P'_s(a)-Q''_s(a)n_{\alpha_1})m_{\beta_1}\ldots n_{\alpha_{r'}}m_{\beta_{r'}}n_{\alpha_r}m_{\beta_r}\ldots n_{\alpha_1}m_{\beta_1}b\\
&&-b\mathfrak q_{1k}(a)+\sum_{s+t+\lambda=k}bn_s(P_t(a)m_{\lambda}-m_{\lambda}Q''_t(a))\\
&&+b\sum_{i+j=k}\sum_{r=1}^{\eta_j}\sum_{(\alpha+\beta)_r=j}\sum_{r'=1}^{\eta_i}\sum_{s+(\alpha+\beta)_{r'}=i}n_{\alpha_r}m_{\beta_r}\ldots n_{\alpha_1}m_{\beta_1}n_{\alpha_{r'}}m_{\beta_{r'}}\ldots n_{\alpha_1}(P_s(a)m_{\beta_1}-m_{\beta_1}Q''_s(a))\\
&=&\big(\mathfrak q_{1k}(a)-\sum_{r=1}^{\eta_k}\sum_{i+(\alpha+\beta)_r=k,\, i\leq k-2}(n_{\alpha_1}P'_i(a)-Q''_i(a)n_{\alpha_1})m_{\beta_1}\ldots n_{\alpha_r}m_{\beta_r}\big)b\\
&&-b\big(\mathfrak q_{1k}(a)-\sum_{r=1}^{\eta_k}\sum_{i+(\alpha+\beta)_r=k,\, i\leq k-2}n_{\alpha_r}m_{\beta_r}\ldots n_{\alpha_1}(P_i(a)m_{\beta_1}-m_{\beta_1}Q''_i(a))\big)\\
&=&[Q''_k(a),b]
\end{eqnarray*}
i.e. $Q''_k(a):=\mathfrak q_{1k}(a)-q''_k(a)\in Z(B)$ for all $a\in A$. By similar argument one can check that\\
$P''_k(b):=\mathfrak p_{2k}(b)-p''_k(b)\in Z(A)$ for all $b\in B$.
Now apply $\K_k$ on commutator
$\left[\left(\begin{array}{cc}
0 & 0\\
0 & b
\end{array}\right),\left(\begin{array}{cc}
0 & 0\\
0 & b'
\end{array}\right)\right]$ we have
\begin{eqnarray}\label{6Lie}
\left(\begin{array}{cc}
\mathfrak p_{2k}[b,b'] & *\\
* & \mathfrak q_{2k}[b,b']
\end{array}\right)&=&\sum_{i+j=k}\left[\left(\begin{array}{cc}
\mathfrak p_{2i}(b) & \mathfrak f_{2i}(b)\\
\mathfrak g_{2i}(b) & \mathfrak q_{2i}(b)
\end{array}\right),\left(\begin{array}{cc}
\mathfrak p_{2j}(b') & \mathfrak f_{2j}(b')\\
\mathfrak g_{2j}(b') & \mathfrak q_{2j}(b')
\end{array}\right)\right].
\end{eqnarray}
From $(1,1)$-entry of above equation and assumption of induction we have
\begin{eqnarray}\label{Pk2}
\nonumber \mathfrak p_{2k}[b,b']&=&\sum_{i+j=k}[\mathfrak p_{2i}(b),\mathfrak p_{2j}(b')]+\sum_{i+j=k}\big(\mathfrak f_{2i}(b)\mathfrak g_{2j}(b')-\mathfrak f_{2j}(b')\mathfrak g_{2i}(b)\big)\\
\nonumber &=&\sum_{i+j=k}[p''_i(b),p''_j(b')]+\sum_{\alpha_1+\beta_1+\xi+\zeta=k}\big(P''_{\zeta}(b)m_{\beta_1}-m_{\beta_1}Q_{\zeta}(b)\big)\big(n_{\alpha_1}P''_{\xi}(b')-Q'_{\xi}(b')n_{\alpha_1}\big)\\
\nonumber &&-\sum_{\alpha_1+\beta_1+\xi+\zeta=k}\big(P''_{\xi}(b')m_{\beta_1}-m_{\beta_1}Q_{\xi}(b')\big)\big(n_{\alpha_1}P''_{\zeta}(b)-Q'_{\zeta}(b)n_{\alpha_1}\big)\\
\nonumber &=&\sum_{i+j=k}\sum_{r'=1}^{\eta_i}\sum_{\zeta+(\alpha+\beta)_{r'}=i}\sum_{r=1}^{\eta_j}\sum_{\xi+(\alpha+\beta)_r=j}\\
\nonumber &&\big(m_{\beta_1}Q_{\zeta}(b)n_{\alpha_1}\ldots m_{\beta_{r'}}n_{\alpha_{r'}}(m_{\beta_1}Q_{\xi}(b')-P''_{\xi}(b')m_{\beta_1})n_{\alpha_1}\ldots m_{\beta_r}n_{\alpha_r}\\
\nonumber &&-P''_{\zeta}(b)m_{\beta_1}n_{\alpha_1}\ldots m_{\beta_{r'}}n_{\alpha_{r'}}(m_{\beta_1}Q_{\xi}(b')-P''_{\xi}(b')m_{\beta_1})n_{\alpha_1}\ldots m_{\beta_r}n_{\alpha_r}\\
\nonumber &&-m_{\beta_1}Q_{\xi}(b')n_{\alpha_1}\ldots m_{\beta_r}n_{\alpha_r}(m_{\beta_1}Q_{\zeta}(b)-P''_{\zeta}(b)m_{\beta_1})n_{\alpha_1}\ldots m_{\beta_{r'}}n_{\alpha_{r'}}\\
\nonumber &&+P''_{\xi}(b')m_{\beta_1}n_{\alpha_1}\ldots m_{\beta_r}n_{\alpha_r}(m_{\beta_1}Q_{\zeta}(b)-P''_{\zeta}(b)m_{\beta_1})n_{\alpha_1}\ldots m_{\beta_{r'}}n_{\alpha_{r'}}\big)\\
\nonumber &&+\sum_{\alpha_1+\beta_1+\xi+\zeta=k}m_{\beta_1}Q_{\zeta}(b)Q'_{\xi}(b')n_{\alpha_1}-\sum_{\alpha_1+\beta_1+\xi+\zeta=k}m_{\beta_1}Q_{\xi}(b')Q'_{\zeta}(b)n_{\alpha_1}\\
\nonumber &&+\sum_{\alpha_1+\beta_1+\xi+\zeta=k}m_{\beta_1}Q_{\xi}(b')n_{\alpha_1}P''_{\zeta}(b)-\sum_{\alpha_1+\beta_1+\xi+\zeta=k}P''_{\zeta}(b)m_{\beta_1}Q'_{\xi}(b')n_{\alpha_1}\\
\nonumber &&+\sum_{\alpha_1+\beta_1+\xi+\zeta=k}P''_{\xi}(b')m_{\beta_1}Q'_{\zeta}(b)n_{\alpha_1}-\sum_{\alpha_1+\beta_1+\xi+\zeta=k}m_{\beta_1}Q_{\zeta}(b)n_{\alpha_1}P''_{\xi}(b')\\
\nonumber &=&\sum_{i+j=k}\sum_{r=1}^{\eta_j}\sum_{\zeta+(\alpha+\beta)_r=j}\big(m_{\beta_1}Q_{\zeta}(b)q_i(b')n_{\alpha_1}\ldots m_{\beta_r}n_{\alpha_r}-P''_{\zeta}(b)m_{\beta_1}q_i(b')n_{\alpha_1}\ldots m_{\beta_r}n_{\alpha_r}\big)\\
\nonumber &&-\sum_{i+j=k}\sum_{r'=1}^{\eta_i}\sum_{\xi+(\alpha+\beta)_{r'}=i}m_{\beta_1}Q_{\xi}(b')q_j(b)n_{\alpha_1}\ldots m_{\beta_{r'}}n_{\alpha_{r'}}\\
\nonumber &&+\sum_{i+j=k}\sum_{r'=1}^{\eta_i}\sum_{\xi+(\alpha+\beta)_{r'}=i}P''_{\xi}(b')m_{\beta_1}q_j(b)n_{\alpha_1}\ldots m_{\beta_{r'}}n_{\alpha_{r'}}\\
\nonumber &&+\sum_{i+j=k}\sum_{\alpha_1+\beta_1+\zeta=j}m_{\beta_1}Q_{\zeta}(b)q_{2i}(b')n_{\alpha_1}-\sum_{i+j=k}\sum_{\alpha_1+\beta_1+\xi=i}m_{\beta_1}Q_{\xi}(b')q_{2j}(b)n_{\alpha_1}\\
\nonumber &&+\sum_{i+j=k}\sum_{\alpha_1+\beta_1+\zeta=j}m_{\beta_1}Q_{\zeta}(b)q'_i(b')n_{\alpha_1}-\sum_{i+j=k}\sum_{\alpha_1+\beta_1+\xi=i}m_{\beta_1}Q_{\xi}(b')q''_{2j}(b)n_{\alpha_1}\\
\nonumber &&+\sum_{\alpha_1+\beta_1+\xi+\zeta=k}m_{\beta_1}Q_{\xi}(b')n_{\alpha_1}P''_{\zeta}(b)-\sum_{\alpha_1+\beta_1+\xi+\zeta=k}P''_{\zeta}(b)m_{\beta_1}Q'_{\xi}(b')n_{\alpha_1}\\
&&+\sum_{\alpha_1+\beta_1+\xi+\zeta=k}P''_{\xi}(b')m_{\beta_1}Q'_{\zeta}(b)n_{\alpha_1}-\sum_{\alpha_1+\beta_1+\xi+\zeta=k}m_{\beta_1}Q_{\zeta}(b)n_{\alpha_1}P''_{\xi}(b').
\end{eqnarray}
By replacing $Q'_*$ with $\mathfrak q_{2*}+q'_*$ in following sentences of relation \eqref{Pk2} we have
\vspace*{-.25cm}
{\footnotesize \begin{eqnarray*}
&&\sum_{i+j=k}\sum_{r=1}^{\eta_j}\sum_{\zeta+(\alpha+\beta)_r=j}m_{\beta_1}Q_{\zeta}(b)q_i(b')n_{\alpha_1}(m_{\beta}n_{\alpha})^r\\
&&+\sum_{i+j=k}\sum_{\alpha_1+\beta_1+\zeta=j}m_{\beta_1}Q_{\zeta}(b)\mathfrak q_{2i}(b')n_{\alpha_1}
+\sum_{i+j=k}\sum_{\alpha_1+\beta_1+\zeta=j}m_{\beta_1}Q_{\zeta}(b)q'_i(b')n_{\alpha_1}\\
&=&\sum_{i+j=k}\sum_{r=2}^{\eta_j}\sum_{\zeta+(\alpha+\beta)_r=j}m_{\beta_1}Q_{\zeta}(b)q_i(b')n_{\alpha_1}(m_{\beta}n_{\alpha})^r
+\sum_{\alpha_1+\beta_1+\zeta+\xi=k}m_{\beta_1}Q_{\zeta}(b)Q_{\xi}(b')n_{\alpha_1}\\
&&+\sum_{i+j=k}\sum_{\alpha_1+\beta_1+\zeta=j}\sum_{r'=1}^{\eta_i}\sum_{(\alpha+\beta)_{r'}+i_1=i}m_{\beta_1}Q_{\zeta}(b)(Q'_{i_1}(b')n_{\alpha_1}-n_{\alpha_1}P''_{i_1}(b'))m_{\beta_1}\ldots n_{\alpha_{r'}}m_{\beta_{r'}})n_{\alpha_1}\\
&=&\sum_{i+j=k}\sum_{r=2}^{\eta_j}\sum_{\zeta+(\alpha+\beta)_r=j}m_{\beta_1}Q_{\zeta}(b)q_i(b')n_{\alpha_1}(m_{\beta}n_{\alpha})^r
+\sum_{\alpha_1+\beta_1+\zeta+\xi=k}m_{\beta_1}Q_{\zeta}(b)Q_{\xi}(b')n_{\alpha_1}\\
&&+\sum_{i_1+j_1=k}\sum_{r'=1}^{\eta_{j_1}}\sum_{\zeta+(\alpha+\beta)_{r'}=j_1}m_{\beta_1}Q_{\zeta}(b)(\mathfrak q_{2i_1}(b')+q'_{i_1}(b'))n_{\alpha_1}(m_{\beta}n_{\alpha})^{r'}\\
&&-\sum_{i_1+j_1=k}\sum_{r'=1}^{\eta_{j_1}}\sum_{\zeta+(\alpha+\beta)_{r'}=j_1}m_{\beta_1}Q_{\zeta}(b)n_{\alpha_1}P''_{i_1}(b')(m_{\beta}n_{\alpha})^{r'}\\
&=&\sum_{i+j=k}\sum_{r=4}^{\eta_j}\sum_{\zeta+(\alpha+\beta)_r=j}m_{\beta_1}Q_{\zeta}(b)q_i(b')n_{\alpha_1}(m_{\beta}n_{\alpha})^r
+\sum_{(\alpha+\beta)_2+\zeta+\xi=k}m_{\beta_1}Q_{\zeta}(b)Q_{\xi}(b')n_{\alpha_1}m_{\beta_2}n_{\alpha_2}\\
&&+\sum_{i_1+j_1=k}\sum_{r'=1}^{\eta_{j_1}}\sum_{\zeta+(\alpha+\beta)_{r'}=j_1}m_{\beta_1}Q_{\zeta}(b)q'_{i_1}(b')n_{\alpha_1}(m_{\beta}n_{\alpha})^{r'}\\
&&-\sum_{i_1+j_1=k}\sum_{r'=1}^{\eta_{j_1}}\sum_{\zeta+(\alpha+\beta)_{r'}=j_1}m_{\beta_1}Q_{\zeta}(b)n_{\alpha_1}P_{i_1}(b')(m_{\beta}n_{\alpha})^{r'}\\
&=&\\
&&\vdots\\
&=&\sum_{r'=1}^{\eta_k}\sum_{\xi+\zeta+(\alpha+\beta)_{r'}=k}m_{\beta_1}Q_{\zeta}(b)Q_{\xi}(b')n_{\alpha_1}(m_{\beta}n_{\alpha})^{r'}
-\sum_{r'=2}^{\eta_k}\sum_{\xi+\zeta+(\alpha+\beta)_{r'}=k}m_{\beta_1}Q_{\zeta}(b)n_{\alpha_1}P''_{\xi}(b')(m_{\beta}n_{\alpha})^{r'}.
\end{eqnarray*}}
By a similar way on following sentences of relation \eqref{Pk2} we have
\begin{eqnarray*}
&&-\sum_{i+j=k}\sum_{\alpha_1+\beta_1+\xi=i}m_{\beta_1}Q_{\xi}(b')\mathfrak q_{2j}(b)n_{\alpha_1}\\&&-\sum_{i+j=k}\sum_{r'=1}^{\eta_i}\sum_{\xi+(\alpha+\beta)_{r'}=i}m_{\beta_1}Q_{\xi}(b')q_j(b)n_{\alpha_1}\ldots m_{\beta_{r'}}n_{\alpha_{r'}}
-\sum_{i+j=k}\sum_{\alpha_1+\beta_1+\xi=i}m_{\beta_1}Q_{\xi}(b')q'_j(b)n_{\alpha_1}\\
&=&-\sum_{r'=1}^{\eta_k}\sum_{\xi+\zeta+(\alpha+\beta)_{r'}=k}m_{\beta_1}Q_{\xi}(b')Q_{\zeta}(b)n_{\alpha_1}\ldots m_{\beta_{r'}}n_{\alpha_{r'}}\\
&&+\sum_{r'=2}^{\eta_k}\sum_{\xi+\zeta+(\alpha+\beta)_{r'}=k}m_{\beta_1}Q_{\xi}(b')n_{\alpha_1}P''_{\zeta}(b)\ldots m_{\beta_{r'}}n_{\alpha_{r'}}.
\end{eqnarray*}
Now consider following sentences of relation \eqref{Pk2}
\begin{eqnarray*}
\sum_{i+j=k}\sum_{\alpha_1+\beta_1+\zeta=k}m_{\beta_1}Q_i(b')n_{\alpha_1}P''_{\zeta}(b)&&-\sum_{i+j=k}\sum_{r=1}^{\eta_j}\sum_{\zeta+(\alpha+\beta)_r=j}P''_{\zeta}(b)m_{\beta_1}q_i(b')n_{\alpha_1}\ldots m_{\beta_r}n_{\alpha_r}\\
&&-\sum_{i+j=k}\sum_{\alpha_1+\beta_1+\zeta=j}P''_{\zeta}(b)m_{\beta_1}Q'_i(b')n_{\alpha_1}\\
&=&-\sum_{i+j=k}\sum_{r=2}^{\eta_j}\sum_{\zeta+(\alpha+\beta)_r=j}P''_{\zeta}(b)m_{\beta_1}q_i(b')n_{\alpha_1}\ldots m_{\beta_r}n_{\alpha_r}\\
&&-\sum_{i+j=k}\sum_{\alpha_1+\beta_1+\zeta=k}P''_{\zeta}(b)m_{\beta_1}q'_i(b')n_{\alpha_1}\\
&=&\\
&&\vdots\\
&=&-\sum_{r=2}^{\eta_k}\sum_{\xi+\zeta+(\alpha+\beta)_{r}=k}P''_{\zeta}(b)m_{\beta_1}Q_{\xi}(b')n_{\alpha_1}\ldots m_{\beta_r}n_{\alpha_r}
\end{eqnarray*}
Similarly consider following sentences of relation \eqref{Pk2} we have
\begin{eqnarray*}
-\sum_{\alpha_1+\beta_1+\xi+\zeta=k}m_{\beta_1}Q_{\zeta}(b)n_{\alpha_1}P''_{\xi}(b')&&+\sum_{i+j=k}\sum_{r'=1}^{\eta_i}\sum_{\xi+(\alpha+\beta)_{r'}=i}P''_{\xi}(b')m_{\beta_1}q_j(b)n_{\alpha_1}\ldots m_{\beta_{r'}}n_{\alpha_{r'}}\\
&&+\sum_{\alpha_1+\beta_1+\xi+\zeta=k}P''_{\xi}(b')m_{\beta_1}Q'_{\zeta}(b)n_{\alpha_1}\\
&=&\sum_{r'=2}^{\eta_k}\sum_{\xi+\zeta+(\alpha+\beta)_{r'}=k}P''_{\xi}(b')m_{\beta_1}Q_{\zeta}(b)n_{\alpha_1}\ldots m_{\beta_{r'}}n_{\alpha_{r'}}
\end{eqnarray*}
Gather above relations and replace in relation \eqref{Pk2}, from this and assumption of induction we have
\begin{eqnarray*}
\mathfrak p_{2k}[b,b']&=&\sum_{r=1}^{\eta_k}\sum_{\zeta+\xi+(\alpha+\beta)_r=k,\, \zeta+\xi\leq k-2}m_{\beta_1}[Q_{\zeta}(b),Q_{\xi}(b')]n_{\alpha_1}\ldots m_{\beta_r}n_{\alpha_r}\\
&=&\sum_{r=1}^{\eta_k}\sum_{i+(\alpha+\beta)_r=k,\, i\leq k-2}m_{\beta_1}Q_i[b,b']n_{\alpha_1}\ldots m_{\beta_r}n_{\alpha_r}\\
&=&\sum_{r=1}^{\eta_k}\sum_{i+(\alpha+\beta)_r=k,\, i\leq k-2}(m_{\beta_1}Q_i[b,b']-P''_i[b,b']m_{\beta_1})n_{\alpha_1}\ldots m_{\beta_r}n_{\alpha_r}.
\end{eqnarray*}
i.e. $P''_k[b,b']:=\mathfrak p_{2k}[b,b']-p''_k[b,b']=0$ for all $b,b'\in B$. Similarly, $Q''_k[a,a']:=\mathfrak q_{1k}[a,a']-q''_k[a,a']=0$ for all $a, a'\in A$.
From $(2,2)$-entry of equation \eqref{6Lie} we have
\[\mathfrak q_{2k}[b,b']=\sum_{i+j=k}[\mathfrak q_{2i}(b),\mathfrak q_{2j}(b')]+\sum_{i+j=k}\big(\mathfrak g_{2i}(b)\mathfrak f_{2j}(b')-\mathfrak g_{2j}(b')\mathfrak f_{2i}(b)\big).\]
Replace $\mathfrak q_{2k}, \mathfrak q_{2i}, \mathfrak q_{2j}$ with $Q_k-q_k, Q_i-q_i, Q_j-q_j$ respectively, then we have
\begin{eqnarray*}
Q_k[b,b']-q_k[b,b']&=&\sum_{i+j=k}[Q_i(b)-q_i(b),Q_j(b')-q_j(b')]\\
&&+\sum_{\alpha_1+\beta_1+\xi+\zeta=k}(n_{\alpha_1}P''_{\zeta}(b)-Q'_{\zeta}(b)n_{\alpha_1})(P''_{\xi}(b')m_{\beta_1}-m_{\beta_1}Q_{\xi}(b'))\\
&&-\sum_{\alpha_1+\beta_1+\xi+\zeta=k}(n_{\alpha_1}P''_{\xi}(b')-Q'_{\xi}(b')n_{\alpha_1})(P''_{\zeta}(b)m_{\beta_1}-m_{\beta_1}Q_{\zeta}(b)).
\end{eqnarray*}
To show that $Q_k$ is a Lie higher derivation by assumption of induction, it is enough to check following equation
\begin{eqnarray*}
q_k[b,b']&=&\sum_{i+j=k}\big([Q_i(b),q_j(b')]+[q_i(b),Q_j(b')]-[q_i(b),q_j(b')]\big)\\
&&+\sum_{\alpha_1+\beta_1+\xi+\zeta=k}\big(n_{\alpha_1}P''_{\zeta}(b)m_{\beta_1}Q_{\xi}(b')+Q'_{\zeta}(b)n_{\alpha_1}P''_{\xi}(b')m_{\beta_1}-Q'_{\zeta}(b)n_{\alpha_1}m_{\beta_1}Q_{\xi}(b')\big)\\
&&-\sum_{\alpha_1+\beta_1+\xi+\zeta=k}\big(n_{\alpha_1}P''_{\xi}(b')m_{\beta_1}Q_{\zeta}(b)+Q'_{\xi}(b')n_{\alpha_1}P''_{\zeta}(b)m_{\beta_1}-Q'_{\xi}(b')n_{\alpha_1}m_{\beta_1}Q_{\zeta}(b)\big).
\end{eqnarray*}
From definition of $q_k$ and equations $Q_k=\mathfrak q_{2k}+q_k$, $Q'_k=\mathfrak q_{2k}+q'_k$ we have
{\footnotesize \begin{eqnarray*}
\sum_{r=1}^{\eta_k}\sum_{i+(\alpha+\beta)_r=k} n_{\alpha_r}m_{\beta_r}\ldots n_{\alpha_1}m_{\beta_1}Q_i[b,b']&=&
\sum_{i+j=k}Q_i(b)\sum_{r=1}^{\eta_j}\sum_{\xi+(\alpha+\beta)_r=j}n_{\alpha_r}m_{\beta_r}\ldots n_{\alpha_1}m_{\beta_1}Q_{\xi}(b')\\
&&-\sum_{i+j=k}Q_i(b)\sum_{r=1}^{\eta_j}\sum_{\xi+(\alpha+\beta)_r=j}n_{\alpha_r}m_{\beta_r}\ldots n_{\alpha_1}P''_{\xi}(b')m_{\beta_1}\\
&&-\sum_{i+j=k}\sum_{r=1}^{\eta_j}\sum_{\xi+(\alpha+\beta)_r=j} n_{\alpha_r}m_{\beta_r}\ldots n_{\alpha_1}m_{\beta_1}Q_{\xi}(b')Q_i(b)\\
&&+\sum_{i+j=k}\sum_{r=1}^{\eta_j}\sum_{\xi+(\alpha+\beta)_r=j} n_{\alpha_r}m_{\beta_r}\ldots n_{\alpha_1}P''_{\xi}(b')m_{\beta_1}Q_i(b)\\
&&+\sum_{i+j=k}\sum_{r'=1}^{\eta_i}\sum_{\zeta+(\alpha+\beta)_{r'}=i} n_{\alpha_{r'}}m_{\beta_{r'}}\ldots n_{\alpha_1}m_{\beta_1}Q_{\zeta}(b)Q_j(b')\\
&&-\sum_{i+j=k}\sum_{r'=1}^{\eta_i}\sum_{\zeta+(\alpha+\beta)_{r'}=i}n_{\alpha_{r'}}m_{\beta_{r'}}\ldots n_{\alpha_1}P''_{\zeta}(b)m_{\beta_1}Q_j(b')\\
&&-\sum_{i+j=k}Q_j(b')\sum_{r'=1}^{\eta_i}\sum_{\zeta+(\alpha+\beta)_{r'}=i}n_{\alpha_{r'}}m_{\beta_{r'}}\ldots n_{\alpha_1}m_{\beta_1}Q_{\zeta}(b)\\
&&+\sum_{i+j=k}Q_j(b')\sum_{r'=1}^{\eta_i}\sum_{\zeta+(\alpha+\beta)_{r'}=i} n_{\alpha_{r'}}m_{\beta_{r'}}\ldots n_{\alpha_1}P''_{\zeta}(b)m_{\beta_1}\\
&&-\sum_{i+j=k}Q_i(b)\sum_{r=1}^{\eta_j}\sum_{\xi+(\alpha+\beta)_r=j}n_{\alpha_r}m_{\beta_r}\ldots n_{\alpha_1}m_{\beta_1}Q_{\xi}(b')\\
&&+\sum_{i+j=k}Q_i(b)\sum_{r=1}^{\eta_j}\sum_{\xi+(\alpha+\beta)_r=j}n_{\alpha_r}m_{\beta_r}\ldots n_{\alpha_1}P''_{\xi}(b')m_{\beta_1}\\
&&+\sum_{i+j=k}\mathfrak q_{2i}(b)\sum_{r=1}^{\eta_j}\sum_{\xi+(\alpha+\beta)_r=j}n_{\alpha_r}m_{\beta_r}\ldots n_{\alpha_1}m_{\beta_1}Q_{\xi}(b')\\
&&-\sum_{i+j=k}\mathfrak q_{2i}(b)\sum_{r=1}^{\eta_j}\sum_{\xi+(\alpha+\beta)_r=j}n_{\alpha_r}m_{\beta_r}\ldots n_{\alpha_1}P''_{\xi}(b')m_{\beta_1}\\
&&+\sum_{i+j=k}Q_j(b')\sum_{r'=1}^{\eta_i}\sum_{\zeta+(\alpha+\beta)_{r'}=i}n_{\alpha_{r'}}m_{\beta_{r'}}\ldots n_{\alpha_1}m_{\beta_1}Q_{\zeta}(b)\\
&&-\sum_{i+j=k}Q_j(b')\sum_{r'=1}^{\eta_i}\sum_{\zeta+(\alpha+\beta)_{r'}=i}n_{\alpha_{r'}}m_{\beta_{r'}}\ldots n_{\alpha_1}P''_{\zeta}(b)m_{\beta_1}\\
&&-\sum_{i+j=k}\mathfrak q_{2j}(b')\sum_{r'=1}^{\eta_i}\sum_{\zeta+(\alpha+\beta)_{r'}=i}n_{\alpha_{r'}}m_{\beta_{r'}}\ldots n_{\alpha_1}m_{\beta_1}Q_{\zeta}(b)\\
&&+\sum_{i+j=k}\mathfrak q_{2j}(b')\sum_{r'=1}^{\eta_i}\sum_{\zeta+(\alpha+\beta)_{r'}=i}n_{\alpha_{r'}}m_{\beta_{r'}}\ldots n_{\alpha_1}P''_{\zeta}(b)m_{\beta_1}\\
&&+\sum_{i+j=k}\sum_{\zeta+\alpha_1=i}\sum_{\xi+\beta_1=j}n_{\alpha_1}P''_{\zeta}(b)m_{\beta_1}Q_{\xi}(b')\\
&&+\sum_{i+j=k}\sum_{\zeta+\alpha_1=i}\sum_{\xi+\beta_1=j}Q'_{\zeta}(b)n_{\alpha_1}P''_{\xi}(b')m_{\beta_1}\\
%\end{eqnarray*}
%\begin{eqnarray*}
\hspace*{4cm}&&-\sum_{i+j=k}\sum_{\zeta+\alpha_1=i}\sum_{\xi+\beta_1=j}Q'_{\zeta}(b)n_{\alpha_1}m_{\beta_1}Q_{\xi}(b')\\
&&-\sum_{i+j=k}\sum_{\zeta+\alpha_1=i}\sum_{\xi+\beta_1=j}n_{\alpha_1}P''_{\xi}(b')m_{\beta_1}Q_{\zeta}(b)\\
&&-\sum_{i+j=k}\sum_{\zeta+\beta_1=i}\sum_{\xi+\alpha_1=j}Q'_{\xi}(b')n_{\alpha_1}P''_{\zeta}(b)m_{\beta_1}\\
&&+\sum_{i+j=k}\sum_{\zeta+\beta_1=i}\sum_{\xi+\alpha_1=j}Q'_{\xi}(b')n_{\alpha_1}m_{\beta_1}Q_{\zeta}(b).
\end{eqnarray*}} 
By omitting similar sentences we have
\begin{eqnarray*}
0&=&\sum_{i+j=k}\sum_{r=2}^{\eta_j}\sum_{\xi+(\alpha+\beta)_r=j} n_{\alpha_r}m_{\beta_r}\ldots n_{\alpha_1}P''_{\xi}(b')m_{\beta_1}Q_i(b)\\
&&-\sum_{i+j=k}\sum_{r'=2}^{\eta_i}\sum_{\zeta+(\alpha+\beta)_{r'}=i}n_{\alpha_{r'}}m_{\beta_{r'}}\ldots n_{\alpha_1}P''_{\zeta}(b)m_{\beta_1}Q_j(b')\\
&&+\sum_{i+j=k}\mathfrak q_{2i}(b)\sum_{r=2}^{\eta_j}\sum_{\xi+(\alpha+\beta)_r=j}n_{\alpha_r}m_{\beta_r}\ldots n_{\alpha_1}m_{\beta_1}Q_{\xi}(b')\\
&&-\sum_{i+j=k}\mathfrak q_{2i}(b)\sum_{r=2}^{\eta_j}\sum_{\xi+(\alpha+\beta)_r=j}n_{\alpha_r}m_{\beta_r}\ldots n_{\alpha_1}P''_{\xi}(b')m_{\beta_1}\\
&&-\sum_{i+j=k}\mathfrak q_{2j}(b')\sum_{r'=2}^{\eta_i}\sum_{\zeta+(\alpha+\beta)_{r'}=i}n_{\alpha_{r'}}m_{\beta_{r'}}\ldots n_{\alpha_1}m_{\beta_1}Q_{\zeta}(b)\\
&&+\sum_{i+j=k}\mathfrak q_{2j}(b')\sum_{r'=2}^{\eta_i}\sum_{\zeta+(\alpha+\beta)_{r'}=i}n_{\alpha_{r'}}m_{\beta_{r'}}\ldots n_{\alpha_1}P''_{\zeta}(b)m_{\beta_1}\\
&&+\sum_{i+j=k}\sum_{\zeta+\alpha_1=i}\sum_{\xi+\beta_1=j}q'_{\zeta}(b)n_{\alpha_1}P''_{\xi}(b')m_{\beta_1}\\
&&-\sum_{i+j=k}\sum_{\zeta+\alpha_1=i}\sum_{\xi+\beta_1=j}q'_{\zeta}(b)n_{\alpha_1}m_{\beta_1}Q_{\xi}(b')\\
&&-\sum_{i+j=k}\sum_{\zeta+\beta_1=i}\sum_{\xi+\alpha_1=j}q'_{\xi}(b')n_{\alpha_1}P''_{\zeta}(b)m_{\beta_1}\\
&&+\sum_{i+j=k}\sum_{\zeta+\beta_1=i}\sum_{\xi+\alpha_1=j}q'_{\xi}(b')n_{\alpha_1}m_{\beta_1}Q_{\zeta}(b)\\
&=&\\
&&\vdots\\
&=&\sum_{(\alpha+\beta)_{\eta_k-1}=k-1}n_{\alpha_{\eta_k-1}}m_{\beta_{\eta_k-1}}\ldots n_{\alpha_1}P''_1(b')m_{\beta_1}b\\
&&-\sum_{(\alpha+\beta)_{\eta_k-1}=k-1}n_{\alpha_{\eta_k-1}}m_{\beta_{\eta_k-1}}\ldots n_{\alpha_1}P''_1(b)m_{\beta_1}b'\\
&&+\sum_{(\alpha+\beta)_{\eta_k}=k}bn_{\alpha_{\eta_k}}m_{\beta_{\eta_k}}\ldots n_{\alpha_1}m_{\beta_1}b'-\sum_{(\alpha+\beta)_{\eta_k}=k}b'n_{\alpha_{\eta_k}}m_{\beta_{\eta_k}}\ldots n_{\alpha_1}m_{\beta_1}b\\
&&-\sum_{(\alpha+\beta)_{\eta_k-1}=k-1}n_{\alpha_1}P''_1(b')m_{\beta_1}\ldots n_{\alpha_{\eta_k-1}}m_{\beta_{\eta_k-1}} b\\
&&+\sum_{(\alpha+\beta)_{\eta_k-1}=k-1}n_{\alpha_1}P''_1(b)m_{\beta_1}\ldots n_{\alpha_{\eta_k-1}}m_{\beta_{\eta_k-1}}b'\\
&&-\sum_{(\alpha+\beta)_{\eta_k}=k}bn_{\alpha_1}m_{\beta_1}\ldots n_{\alpha_{\eta_k}}m_{\beta_{\eta_k}}b'+\sum_{(\alpha+\beta)_{\eta_k}=k}b'n_{\alpha_1}m_{\beta_1}\ldots n_{\alpha_{\eta_k}}m_{\beta_{\eta_k}}b.
\end{eqnarray*}
Hence $\{Q_k\}_{k\in\mathbb{N}_0}$ is a Lie higher derivation on $B$. By similar techniques one can show that $\{P_k\}_{k\in\mathbb{N}_0}, \{P'_k\}_{k\in\mathbb{N}_0}$ are Lie higher derivations on $A$ and also $\{Q'_k\}_{k\in\mathbb{N}_0}$ is a Lie higher derivation on $B$. \\
Now cosider commutatore $\left[\left(\begin{array}{cc}
0 & m\\
0 & 0
\end{array}\right),\left(\begin{array}{cc}
0 & 0\\
n & 0
\end{array}\right)\right]$, apply $\K_k$ on it then we have
\begin{eqnarray*}
\left(\begin{array}{cc}
\mathfrak p_{1k}(mn)-\mathfrak p_{2k}(nm) & *\\
* & \mathfrak q_{1k}(mn)-\mathfrak q_{2k}(nm)
\end{array}\right)&=&\sum_{i+j=k}\left[\K_i\left(\begin{array}{cc}
0 & m\\
0 & 0
\end{array}\right),\K_j\left(\begin{array}{cc}
0 & 0\\
n & 0
\end{array}\right)\right]\\
&=&\sum_{i+j=k}\left[\left(\begin{array}{cc}
\mathfrak p_{3i}(m) & \mathfrak f_{3i}(m)\\
\mathfrak g_{3i}(m) & \mathfrak q_{3i}(m)
\end{array}\right),\left(\begin{array}{cc}
\mathfrak p_{4j}(n) & \mathfrak f_{4j}(n)\\
\mathfrak g_{4j}(n) & \mathfrak q_{4j}(n)
\end{array}\right)\right].
\end{eqnarray*}
Hence \[\mathfrak p_{1k}(mn)-\mathfrak p_{2k}(nm)=\sum_{i+j=k}\big(\mathfrak p_{3i}(m)\mathfrak p_{4j}(n) +\mathfrak f_{3i}(m)\mathfrak g_{4j}(n)-\mathfrak p_{4j}(n)\mathfrak p_{3i}(m)-\mathfrak f_{4j}(n)\mathfrak g_{3i}(m)\big),\]
and \[ \mathfrak q_{1k}(mn)-\mathfrak q_{2k}(nm)=\sum_{i+j=k}\big(\mathfrak g_{3i}(m)\mathfrak f_{4j}(n)+\mathfrak q_{3i}(m)\mathfrak q_{4j}(n)-\mathfrak g_{4j}(n)\mathfrak f_{3i}(m)-\mathfrak q_{4j}(n)\mathfrak q_{3i}(m)\big).\]
\end{proof}
%%%%%%%%%%%%%%%%%%%%%%%%%%%%%%%%%%%%%%%%%%%%%
\subsection*{\color{SEC}Proof of Theorem \ref{HD}}
\begin{proof}
From the fact that ``every higher derivation is a Lie higher derivation" and Proposition \ref{LHD} follow that items $(3), (8), (9)$ and $(10)$ hold automatically.
For other items, we proceed the proof by induction on $k$. The case $k=1$ follows  from \cite{DW}. Suppose that the conclusion is true for any integer less than $k$, and $\K_k$ has the presentation
\[\D_k\left(\begin{array}{cc}
a & m\\
n & b
\end{array}\right)=\left(\begin{array}{cc}
\mathtt p_{1k}(a)+\mathtt p_{2k}(b)+\mathtt p_{3k}(m)+\mathtt p_{4k}(n) & \mathtt f_{1k}(a)+\mathtt f_{2k}(b)+\mathtt f_{3k}(m)+\mathtt f_{4k}(n)\\
\mathtt g_{1k}(a)+\mathtt g_{2k}(b)+\mathtt g_{3k}(m)+\mathtt g_{4k}(n) & \mathtt q_{1k}(a)+\mathtt q_{2k}(b)+\mathtt q_{3k}(m)+\mathtt q_{4k}(n)
\end{array}\right),\]
in which the maps appeared in the entries are linear.\\
Apply $\D_k$ on equation
$\left(\begin{array}{cc}
a & m\\
0 & 0
\end{array}\right)\left(\begin{array}{cc}
0 & 0\\
0 & b
\end{array}\right)=\left(\begin{array}{cc}
0 & mb\\
0 & 0
\end{array}\right)$ then we have
\begin{eqnarray*}
\nonumber\left(\begin{array}{cc}
\mathtt p_{3k}(mb) & \mathtt f_{3k}(mb)\\
\mathtt g_{3k}(mb) & \mathtt q_{3k}(mb)
\end{array}\right)&=&\left(\begin{array}{cc}
\mathtt p_{1k}(a)+\mathtt p_{3k}(m) & \mathtt f_{1k}(a)+ \mathtt f_{3k}(m)\\
\mathtt g_{1k}(a)+\mathtt g_{3k}(m) & \mathtt q_{1k}(a)+\mathtt q_{3k}(m)
\end{array}\right)\left(\begin{array}{cc}
0 & 0\\
0 & b
\end{array}\right)\\
\nonumber &&+\left(\begin{array}{cc}
a & m\\
0 & 0
\end{array}\right)\left(\begin{array}{cc}
\mathtt p_{2k}(b) & \mathtt f_{2k}(b)\\
\mathtt g_{2k}(b) & \mathtt q_{2k}(b)
\end{array}\right)\\
&&+\hspace*{-.5cm}\sum_{i+j=k, i, j\neq 0}\left(\begin{array}{cc}
\mathtt p_{1i}(a)+\mathtt p_{3i}(m) & \mathtt f_{1i}(a)+ \mathtt f_{3i}(m)\\
\mathtt g_{1i}(a)+\mathtt g_{3i}(m) & \mathtt q_{1i}(a)+\mathtt q_{3i}(m)
\end{array}\right)\left(\begin{array}{cc}
\mathtt p_{2j}(b) & \mathtt f_{2j}(b)\\
\mathtt g_{2j}(b) & \mathtt q_{2j}(b)
\end{array}\right).
\end{eqnarray*}
From $(1,2)$-entry of above equation we have
\begin{eqnarray}\label{f_{3k}HD}
\nonumber \mathtt f_{3k}(mb)&=&\mathtt f_{1k}(a)b+\mathtt f_{3k}(m)b+a\mathtt f_{2k}(b)+m\mathtt q_{2k}(b)\\
&&+\sum_{i+j=k, i, j\neq 0}\big((\mathtt p_{1i}(a)+\mathtt p_{3i}(m))\mathtt f_{2j}(b)+(\mathtt f_{1i}(a)+\mathtt f_{3i}(m))\mathtt q_{2j}(b)\big).
\end{eqnarray}
Set $m=0,\, b=1$ in equation \eqref{f_{3k}HD} then we have
\begin{eqnarray*}
\mathtt f_{1k}(a)&=&am_k+\sum_{i+j=k, i, j\neq 0}\big(\mathtt p_{1i}(a)m_j-\mathtt f_{1i}(a)\mathtt q_{2j}(1)\big)\\
&=&\sum_{i+j=k,  j\neq 0}\mathtt p_{1i}(a)m_j+\sum_{i+j=k, i, j\neq 0}\mathtt f_{1i}(a)q_j(1)\\
&=&\sum_{i+j=k,  j\neq 0}\mathtt p_{1i}(a)m_j+\sum_{i+j=k, i, j\neq 0}\mathtt f_{1i}(a)\sum_{r=1}^{\eta_j}\sum_{(\alpha+\beta)_r=j}n_{\alpha_r}m_{\beta_r}\ldots n_{\alpha_1}m_{\beta_1}\\
&=&\sum_{i+j=k,  j\neq 0}\mathtt p_{1i}(a)m_j+\sum_{i+j=k, i, j\neq 0}\sum_{r=1}^{\eta_j}\sum_{(\alpha+\beta)_r=j}\sum_{t+s=i}\mathsf P_t(a)m_sn_{\alpha_r}m_{\beta_r}\ldots n_{\alpha_1}m_{\beta_1}.\\
&=&\sum_{i+j=k,  j\neq 0}\mathtt p_{1i}(a)m_j+\sum_{i+j=k, i, j\neq 0}\sum_{r=1}^{\eta_i}\sum_{t+(\alpha+\beta)_r=i}\mathsf P_t(a)m_{\beta_1}n_{\alpha_1}\ldots m_{\beta_r} n_{\alpha_r}m_j.\\
&=&\sum_{i+j=k}\mathsf P_i(a)m_j.
\end{eqnarray*}
Similarly one can check that similar equations for $\mathtt g_{1k}(a), \mathtt g_{2k}(b), \mathtt f_{2k}(b)$.
Now, set $m=0$ in \eqref{f_{3k}HD} hence
\begin{eqnarray}\label{f_3(m)H}
\nonumber \mathtt f_{3k}(mb)&=&\mathtt f_{3k}(m)b+m\mathtt q_{2k}(b)\\
\nonumber &&+\sum_{i+j=k,\, i, j\neq 0}\big(\mathtt p_{3i}(m)\mathtt f_{2j}(b)+\mathtt f_{3i}(m)\mathtt q_{2j}(b)\big)\\
&=&\sum_{i+j=k}\mathtt f_{3i}(m)\mathtt q_{2j}(b)+\sum_{i+j=k,\, i, j\neq 0}\mathtt p_{3i}(m)\mathtt f_{2j}(b).
\end{eqnarray}
As
\begin{eqnarray*}
\mathtt p_{3i}(m)\mathtt f_{2j}(b)&=&\sum_{s+t=i}\sum_{\lambda+\mu=j}\mathtt f_{3t}(m)\NN_sm_{\lambda}\mathsf Q_{\mu}(b)\\
&=&\sum_{s+t=i}\sum_{\lambda+\mu=j}\sum_{r=1}^{\nu_s}\sum_{(\alpha+\beta)_r+\gamma=s}
\mathtt f_{3t}(m)(n_{\alpha_1}m_{\beta_1}\ldots n_{\alpha_r}m_{\beta_r}n_{\gamma}+n_s)m_{\lambda}\mathsf Q_{\mu}(b)\\
&=&\mathtt f_{3i}(m)\sum_{r=1}^{\eta_j}\sum_{(\alpha+\beta)_r+\mu =j, \mu \leq k-2}n_{\alpha_1}m_{\beta_1}\ldots n_{\alpha_r}m_{\beta_r}\mathsf Q_{\mu}(b)\\
&=&\mathtt f_{3i}(m)\mathsf q_j(b),
\end{eqnarray*}
we can replace this relation in \eqref{f_3(m)H} and get
\[\mathtt f_{3k}(mb)=\sum_{i+j=k}\mathtt f_{3i}(m)\mathsf Q_j(b).\]
Similarly one can check similar equations for $\mathtt f_{3k}(am), \mathtt g_{4k}(na), \mathtt g_{4k}(bn)$.
Now apply $\D_k$ on equation
$\left(\begin{array}{cc}
0 & 0\\
0 & b
\end{array}\right)\left(\begin{array}{cc}
0 & 0\\
0 & b'
\end{array}\right)=\left(\begin{array}{cc}
0 & 0\\
0 & bb'
\end{array}\right)$, then
\begin{eqnarray}\label{7HD}
\left(\begin{array}{cc}
\mathtt p_{2k}(bb')& *\\
* & \mathtt q_{2k}(bb')
\end{array}\right)&=&\sum_{i+j=k}\left(\begin{array}{cc}
\mathtt p_{2i}(b) & \mathtt f_{2i}(b)\\
\mathtt g_{2i}(b) & \mathtt q_{2i}(b)
\end{array}\right)\left(\begin{array}{cc}
\mathtt p_{2j}(b') & \mathtt f_{2j}(b')\\
\mathtt g_{2j}(b') & \mathtt q_{2j}(b')
\end{array}\right).
\end{eqnarray}
From $(1,1)$-entry of above equation and assumption of induction we have
\small \begin{eqnarray*}
\mathtt p_{2k}(bb')&=&\sum_{i+j=k}\big(\mathtt p_{2i}(b)\mathtt p_{2j}(b')+\mathtt f_{2i}(b)\mathtt g_{2j}(b')\big)\\
&=&\sum_{i+j=k}\sum_{r=1}^{\eta_i}\sum_{\zeta+(\alpha+\beta)_r=i}\sum_{r'=1}^{\eta_j}\sum_{\xi+(\alpha+\beta)_{r'}=j}
m_{\beta_1}\mathsf Q_{\zeta}(b)n_{\alpha_1}\ldots m_{\beta_r}n_{\alpha_r}m_{\beta_{r'}}n_{\alpha_{r'}}\ldots m_{\beta_1}\mathsf Q'_{\xi}(b')n_{\alpha_1}\\
&&+\sum_{i+j=k}\sum_{\zeta+\beta_1=i}\sum_{\xi+\alpha_1=j}m_{\beta_1}\mathsf Q_{\zeta}(b)\mathsf Q'_{\xi}(b')n_{\alpha_1}.
\end{eqnarray*}
Set $b'=1$ it follows that
\begin{eqnarray*}
\mathtt p_{2k}(b)&=&\sum_{i+j=k}\sum_{r=1}^{\eta_i}\sum_{\zeta+(\alpha+\beta)_r=i}\sum_{r'=1}^{\eta_j}\sum_{(\alpha+\beta)_{r'}=j}
m_{\beta_1}\mathsf Q_{\zeta}(b)n_{\alpha_1}\ldots m_{\beta_r}n_{\alpha_r}m_{\beta_1}n_{\alpha_1}\ldots m_{\beta_{r'}}n_{\alpha_{r'}}\\
&&+\sum_{i+j=k}\sum_{\zeta+\beta_1=i}m_{\beta_1}\mathsf Q_{\zeta}(b)n_{\alpha_j}\\
&=&\sum_{r=2}^{\eta_k}\sum_{\zeta+(\alpha+\beta)_r=k}m_{\beta_1}\mathsf Q_{\zeta}(b)n_{\alpha_1}\ldots m_{\beta_r}n_{\alpha_r}+\sum_{\zeta+\alpha_1+\beta_1=k}m_{\beta_1}\mathsf Q_{\zeta}(b)n_{\alpha_1}\\
&=&\sum_{r=1}^{\eta_k}\sum_{\zeta+(\alpha+\beta)_r=k}m_{\beta_1}\mathsf Q_{\zeta}(b)n_{\alpha_1}\ldots m_{\beta_r}n_{\alpha_r}, \qquad(b\in B).\\
\end{eqnarray*}
Note that if we set $b=1$ then
\[\mathtt p_{2k}(b')=\sum_{r=1}^{\eta_k}\sum_{\zeta+(\alpha+\beta)_r=k}m_{\beta_r}n_{\alpha_r}\ldots m_{\beta_1}\mathsf Q'_{\zeta}(b')n_{\alpha_1}, \qquad(b'\in B).\]
A similar argument reveals that
\begin{eqnarray*}
\mathtt q_{1k}(a)&=&\sum_{r=1}^{\eta_k}\sum_{i+(\alpha+\beta)_r=k,\, i\leq k-2}n_{\alpha_r}m_{\beta_r}\ldots n_{\alpha_1}\mathsf P_i(a)m_{\beta_1}\\
&=&\sum_{r=1}^{\eta_k}\sum_{i+(\alpha+\beta)_r=k,\, i\leq k-2}n_{\alpha_1}\mathsf P'_i(a)m_{\beta_1}\ldots n_{\alpha_r}m_{\beta_r},\qquad (a\in A)
\end{eqnarray*}
From $(2,2)$-entry of equation \eqref{7HD} we have
\[\mathtt q_{2k}(bb')=\sum_{i+j=k}\big(\mathtt g_{2i}(b)\mathtt f_{2j}(b')+\mathtt q_{2i}(b)\mathtt q_{2j}(b')\big).\]
By using assumption of induction and replacement $\mathtt q_{2k}, \mathtt q_{2i}, \mathtt q_{2j}$ with $\mathsf Q_k-\mathsf q_k, \mathsf Q_i-\mathsf q_i, \mathsf Q_j-\mathsf q_j$ respectively, we have
\begin{eqnarray*}
\mathsf Q_k(bb')-\mathsf q_k(bb')&=&\sum_{i+j=k}(\mathsf Q_i(b)-\mathsf q_i(b))(\mathsf Q_j(b')-\mathsf q_j(b'))\\
&&+\sum_{i+j=k,\, i, j\neq 0}\sum_{\xi+\alpha_1=i}\sum_{\zeta+\beta_1=j}\mathsf Q'_{\xi}(b)n_{\alpha_1}m_{\beta_1}\mathsf Q_{\zeta}(b').
\end{eqnarray*}
In sequal we show that
\begin{eqnarray*}
\mathsf q_k(bb')&=&\sum_{i+j=k}\big(\mathsf Q_i(b)\mathsf q_j(b')+\mathsf q_i(b)\mathsf Q_j(b')-\mathsf q_i(b)\mathsf q_j(b')\big)\\
&&-\sum_{i+j=k,\, i, j\neq 0}\sum_{\xi+\alpha_1=i}\sum_{\zeta+\beta_1=j}\mathsf Q'_{\xi}(b)n_{\alpha_1}m_{\beta_1}\mathsf Q_{\zeta}(b').
\end{eqnarray*}
As
\begin{eqnarray*}
\sum_{r=1}^{\eta_k}\sum n_{\alpha_r}m_{\beta_r}\ldots n_{\alpha_1}m_{\beta_1}\mathsf Q_i(bb')&=&
\sum_{i+j=k}\mathsf Q_i(b)\sum_{r=1}^{\eta_j}\sum_{\zeta+(\alpha+\beta)_r=j} n_{\alpha_r}m_{\beta_r}\ldots n_{\alpha_1}m_{\beta_1}\mathsf Q_{\zeta}(b')\\
&&+\sum_{i+j=k}\sum_{r'=1}^{\eta_i}\sum_{\xi+(\alpha+\beta)_{r'}=i} n_{\alpha_{r'}}m_{\beta_{r'}}\ldots n_{\alpha_1}m_{\beta_1}\mathsf Q_{\xi}(b)\mathsf Q_j(b')\\
&&-\sum_{i+j=k}\mathsf Q_i(b)\sum_{r=1}^{\eta_j}\sum_{\zeta+(\alpha+\beta)_r=j}n_{\alpha_r}m_{\beta_r}\ldots n_{\alpha_1}m_{\beta_1}\mathsf Q_{\zeta}(b')\\
&&+\sum_{i+j=k}\mathtt q_{2i}(b)\sum_{r=1}^{\eta_j}\sum_{\zeta+(\alpha+\beta)_r=j}n_{\alpha_r}m_{\beta_r}\ldots n_{\alpha_1}m_{\beta_1}\mathsf Q_{\zeta}(b')\\
&&-\sum_{i+j=k}\sum_{\xi+\alpha_1=i}\sum_{\zeta+\beta_1=j}\mathsf Q'_{\xi}(b)n_{\alpha_1}m_{\beta_1}\mathsf Q_{\zeta}(b').
\end{eqnarray*}
By assumption of induction, cancelling similar sentences and replacin $\mathsf Q'_i$ with $\mathtt q_{2i}+\mathsf q'_i$ we have
\begin{eqnarray*}
0&=&+\sum_{i+j=k}\mathtt q_{2i}(b)\sum_{r=1}^{\eta_j}\sum_{\zeta+(\alpha+\beta)_r=j}n_{\alpha_r}m_{\beta_r}\ldots n_{\alpha_1}m_{\beta_1}\mathsf Q_{\zeta}(b')\\
&&-\sum_{i+j=k}\mathsf Q'_i(b)\sum_{\zeta+\alpha_1+\beta_1=j}n_{\alpha_1}m_{\beta_1}\mathsf Q_{\zeta}(b')\\
&=&+\sum_{i+j=k}\mathtt q_{2i}(b)\sum_{r=2}^{\eta_j}\sum_{\zeta+(\alpha+\beta)_r=j}n_{\alpha_r}m_{\beta_r}\ldots n_{\alpha_1}m_{\beta_1}\mathsf Q_{\zeta}(b')\\
&&-\sum_{i+j=k}\mathsf q'_i(b)\sum_{\zeta+\alpha_1+\beta_1=j}n_{\alpha_1}m_{\beta_1}\mathsf Q_{\zeta}(b')\\
&=&\\
&&\vdots\\
&=&\sum_{(\alpha+\beta)_{\eta_k}=k}bn_{\alpha_{\eta_k}}m_{\beta_{\eta_k}}\ldots n_{\alpha_1}m_{\beta_1}b'-\sum_{(\alpha+\beta)_{\eta_k}=k}bn_{\alpha_{\eta_k}}m_{\beta_{\eta_k}}\ldots n_{\alpha_1}m_{\beta_1}b'.
\end{eqnarray*}
Hence $\{\mathsf Q_k\}_{k\in\mathbb{N}_0}$ is a higher derivation on $B$. By similar techniques one can be shown that
$\{\mathsf P_k\}_{k\in\mathbb{N}_0}, \{\mathsf P'_k\}_{k\in\mathbb{N}_0}$ are higher derivations on $A$ and also $\{\mathsf Q'_k\}_{k\in\mathbb{N}_0}$ is a higher derivation on $B$.\\
Now consider equation $\left(\begin{array}{cc}
0 & m\\
0 & 0
\end{array}\right)\left(\begin{array}{cc}
0 & 0\\
n & 0
\end{array}\right)=\left(\begin{array}{cc}
mn & 0\\
0 & 0
\end{array}\right)$, apply $\D_k$ on it then we have,
\begin{eqnarray*}
\left(\begin{array}{cc}
\mathtt p_{1k}(mn) & *\\
* & \mathtt q_{1k}(mn)
\end{array}\right)&=&\sum_{i+j=k}\left(\begin{array}{cc}
\mathtt p_{3i}(m) & \mathtt f_{3i}(m)\\
\mathtt g_{3i}(m) & \mathtt q_{3i}(m)
\end{array}\right)\left(\begin{array}{cc}
\mathtt p_{4j}(n) &\mathtt  f_{4j}(n)\\
\mathtt g_{4j}(n) & \mathtt q_{4j}(n)
\end{array}\right).
\end{eqnarray*}
Hence
\[\mathtt p_{1k}(mn)=\sum_{i+j=k}\big(\mathtt p_{3i}(m)\mathtt p_{4j}(n)+\mathtt f_{3i}(m)\mathtt g_{4j}(n)\big),\]
and
\[\mathtt q_{1k}(mn)=\sum_{i+j=k}\big(\mathtt g_{3i}(m)\mathtt  f_{4j}(n)+\mathtt q_{3i}(m) \mathtt q_{4j}(n)\big).\]
Similarly one can check similar equations for $\mathtt p_{2k}(nm)$ and $\mathtt q_{2k}(nm)$.
\end{proof}
\subsection*{\color{SEC} Proof of Proposition \ref{center}}
\begin{proof}
Let $\tau$ maps into the center of $\mathcal{G}$ and vanishes at commutators. Suppose that linear map $\tau_k:\mathcal{G}\longrightarrow Z(\mathcal{G})$ has general form as;
\small \begin{eqnarray}\label{sF1}
\tau_k\left(\begin{array}{cc}
a & m\\
n & b
\end{array}\right)=\left(\begin{array}{cc}
p_{1k}(a)+p_{2k}(b)+p_{3k}(m)+p_{4k}(n) &  \\
  & q_{1k}(b)+q_{2k}(b)+q_{3k}(m)+q_{4k}(n)
\end{array}\right),
\end{eqnarray}
for all $k\in\mathbb{N}$.
As $\tau$ vanishing at commutators we have
\begin{eqnarray}\label{sF2}
\nonumber 0&=&\tau_k\left[\left(\begin{array}{cc}
a & m\\
n & b
\end{array}\right),\left(\begin{array}{cc}
a' & m'\\
n' & b'
\end{array}\right)\right]\\
&=&\tau_k\left(\begin{array}{cc}
[a,a']+mn'-m'n & am'+mb'-a'm-m'b\\
na'+bn'-n'a-b'n & [b,b']+nm'-n'm
\end{array}\right),
\end{eqnarray}
for all $k\in\mathbb{N}$. Now by relations \eqref{sF1} and \eqref{sF2} we have
\[p_{3k}(am'+mb'-a'm-m'b)=0, \quad q_{3k}(am'+mb'-a'm-m'b)=0,\]
if we set $a=a'=0, b=0, b'=1$ in the above commutator we have $p_{3k}(m)=0$ and $q_{3k}(m)=0$ for all $m\in M$ and $k\in\mathbb{N}$. Similarly, $p_{4k}(n)=0, q_{4k}(n)=0$ for all $n\in N, k\in\mathbb{N}$. Moreover if $m'=0, n=0$ we have $p_{1k}(mn')=p_{2k}(n'm), q_{1k}(mn')=q_{2k}(n'm)$ for all $m\in M, n'\in N, k\in\mathbb{N}$. One more time set $b=0, m=0, n=0$ in relation \eqref{sF2} then we have $p_{1k}[a,a']=0$ and $q_{1k}[a,a']=0$ for all $a, a'\in A, k\in\mathbb{N}$. By a similar way one can check that $p_{2k}[b,b']=0, q_{2k}[b,b']=0$ for all $b, b'\in B, k\in\mathbb{N}$. As $\tau$ maps into $Z(\mathcal{G})$ we have;
\[\left[\left(\begin{array}{cc}
p_{1k}(a)+p_{2k}(b) & 0\\
0 & q_{1k}(a)+q_{2k}(b)
\end{array}\right),\left(\begin{array}{cc}
a' & m'\\
n' & b'
\end{array}\right)\right]=0,\]
for all $a,a'\in A, b,b'\in B, m'\in M, n'\in N$ and $k\in\mathbb{N}$. From this, a direct  verification reveals that the remainders hold. The reverse argument is trivial.
\end{proof}

\end{document}